\documentclass[preprint,12pt]{elsarticle}

\usepackage{amssymb}
\usepackage{algorithmic}
\usepackage{amsmath}
\usepackage{caption}
\usepackage{mathtools}
\usepackage{hyperref}
\usepackage{cleveref}
\usepackage{float}
\usepackage{amsthm}
\newtheorem{theorem}{Theorem}[section]
\newtheorem{claim}{Claim}
\usepackage{algorithm}

\usepackage{tabularx}
\usepackage{amssymb}
\usepackage{graphicx}
\usepackage{algorithm}
\usepackage{makecell}
\usepackage{float}
\usepackage{booktabs}
\newtheorem{lemma}{Lemma}
\renewcommand{\vec}[1]{\mathbf{#1}}
\newcommand{\mat}[1]{\mathbf{#1}}

\newtheorem{remark}{Remark}[section]

\hypersetup{
    colorlinks=false,
    pdfnewwindow=true
}

\begin{document}

\begin{frontmatter}




\title{An adaptive split--combine Gaussian mixture filter for nonlinear and multimodal state estimation}

\author[label1]{San Kim}
\author[label2,label3]{Won Chang}
\author[label4,label5]{Daniel B. Forger}
\author{Dae Wook Kim\corref{cor1}\fnref{label1,label6}}
\ead{daewook@kaist.ac.kr}
\cortext[cor1]{Corresponding author}

\affiliation[label1]{organization={Department of Brain and Cognitive Sciences, KAIST},
            city={Daejeon},
            postcode={34141},
            country={Republic\! of\! Korea}}

\affiliation[label2]{organization={Department of Statistics, Seoul National University},
            city={Seoul},
            postcode={08826},
            country={Republic\! of\! Korea}}

\affiliation[label3]{organization={The Institute for Data Innovation in Science, Seoul National University},
            city={Seoul},
            postcode={08826},
            country={Republic\! of\! Korea}}           

\affiliation[label4]{organization={Department of Mathematics, University of Michigan},
            city={Ann Arbor},
            postcode={48109},
            state={MI},
            country={USA}}
            
\affiliation[label5]{organization={Department of Computational Medicine and Bioinformatics},
            city={Ann Arbor},
            postcode={48109},
            state={MI},
            country={USA}}

\affiliation[label6]{organization={KI for Human Augmentation Convergence, KAIST},
            city={Daejeon},
            postcode={34141},
            country={Republic\! of\! Korea}}

\begin{abstract}
Filtering combines model predictions with measurements to estimate the probability density function (PDF) of a system state over time. The PDF often becomes highly asymmetric and even multimodal in nonlinear systems with oscillatory or chaotic dynamics. Such non-Gaussian features violate the single-Gaussian assumption underlying Kalman-type filters. To address this problem, Gaussian mixture filtering has been proposed. However, accurately propagating mixture components and adaptively adjusting their number and weights over time remain open challenges. Here, we develop an adaptive split--combine Gaussian mixture filter (AMF) that estimates the time evolution of asymmetric and multimodal PDFs by adaptively splitting and combining Gaussian particles without auxiliary online numerical optimization. Notably, the proposed splitting method guarantees a reduction in variance along a target level-set-point direction of a Gaussian particle. This enables accurate and efficient propagation of particles. We show that AMF consistently outperforms various baseline filters across diverse benchmarks, including single and coupled slow–fast Van der Pol oscillators and the Lorenz attractor. We also propose a parallel implementation of AMF, allowing high-fidelity PDF estimation with practical computational cost.
\end{abstract}




\end{frontmatter}



\section{Introduction}

Nonlinear dynamics, such as oscillations and chaos, are ubiquitous across physical and biological systems, including atmospheric convection \cite{lorenz2017deterministic}, neuronal activity \cite{hodgkin1952measurement}, cardiac pacemaking \cite{van1928lxxii}, and circadian rhythms \cite{jewett1998refinement}. The full state of these systems, however, often remains inaccessible to direct measurement despite advances in experimental techniques. For example, gene and hormone expression levels in deep brain regions cannot be directly measured in practice \cite{stone2020computational, dijk2020novel, kim2019systems}.

Due to this limitation, estimating unobservable states in stochastic systems is essential. One approach to this problem is recursive Bayesian estimation, i.e., Bayesian filtering \cite{sarkka2023bayesian}. This framework iteratively alternates between two steps: a $time$-$update$ step, which propagates the state between consecutive measurements, and a $measurement$-$update$ step, which corrects the prediction from the time-update step using the latest measurement. Although the optimal Bayesian filter solution can be represented conceptually, its exact computation is generally intractable for nonlinear systems \cite{arasaratnam2009cubature, arasaratnam2010cubature}. Thus, nonlinear filters rely on numerical approximations that can be broadly categorized into two classes. In the first class, the posterior density is constrained to take an $a$ $priori$ form, typically Gaussian. Continuous--discrete Kalman filters, such as the continuous--discrete cubature Kalman filter (CD-CKF) \cite{arasaratnam2010cubature}, belong to this class. In the second class, no assumptions are made about the form of the posterior density, and a representative example is the particle filter (PF) \cite{gordon1993novel, doucet2001introduction}.

Although these methods have been widely used in many applications to date \cite{ristic2003beyond, fearnhead2018particle, khodarahmi2023review}, their accuracy and computational efficiency deteriorate when estimating highly nonlinear dynamic states \cite{sarkka2023bayesian}. For example, CD-CKF employs a 1.5--order It\^o--Taylor expansion of the stochastic differential equation (SDE) to perform the time-update step, and thus its accuracy strongly depends on the system’s local nonlinearity: when the system is highly nonlinear, the probability density function (PDF) of the state variables may exhibit non-Gaussian features such as asymmetry (e.g., skewness) and multimodality \cite{arasaratnam2010cubature}. While non-Gaussian state propagation can be tracked through sequential importance sampling within PF, its performance is constrained by a trade-off between the degree of nonlinearity and the number of samples \cite{daum2003curse, snyder2008obstacles}. These drawbacks become more pronounced when measurements are temporally sparse or provide little information for updating the system states \cite{terejanu2008novel}, particularly in biological oscillator data assimilation \cite{kim2023wearable}.

To circumvent these challenges, Gaussian mixture Filters (GMFs) have been proposed \cite{sorenson1971recursive}. They track the posterior density by propagating a mixture of Gaussian components with non-zero covariance, unlike PF, which uses Dirac delta components. In the GMF paradigm, determining an appropriate number of Gaussian components and their weights is crucial for achieving both accuracy and computational efficiency \cite{shao2019semisupervised}. Specifically, more and finer components are required when the posterior density is complex, whereas fewer and coarser components suffice when it resembles a Gaussian distribution. One approach to this issue is to continuously assess non-Gaussianity and dynamically adjust the number and weights of Gaussian components during filtering \cite{zhang2022gaussian}. However, most existing methods have been developed only for the measurement-update step, not for the time-update step \cite{leong2013gaussian, tuggle2018automated, raitoharju2019partitioned}. Moreover, during the time-update step, the component weights are generally assumed to be fixed \cite{alspach2003nonlinear, ito2002gaussian, chen2000mixture, faubel2009split, Faubel2010Further, Volker2011versatile}. This assumption holds when the covariances of the state variables are sufficiently small, such that the linearization of the system provides a reasonable representation of the dynamics around the mean. However, it easily becomes problematic for highly nonlinear systems or when measurements are noisy or infrequent.

The fixed-weight problem during state propagation has been addressed in only a few studies using optimization-based formulations \cite{terejanu2008novel, terejanu2008uncertainty, singla2008gaussian, Horwood2011Adaptive, Wang2020Adaptive, Li2021Efficient}. For example, Terejanu et al. \cite{terejanu2008novel} incorporated an auxiliary online optimization into the time-update step, where the Gaussian mixture approximation was constrained to satisfy the Fokker--Planck equation governing the time evolution of the state PDF. However, this approach can only be implemented with a fixed number of Gaussian components. Moreover, the numerical optimization problem must be solved at each infinitesimal time step during state propagation between consecutive measurements, which is computationally expensive. It also requires the computation of spatial partial derivatives and numerical integration, which may lead to stability issues.

One possible remedy is to extend adaptive particle split--combine techniques developed for uncertainty propagation \cite{Vittaldev2016spacecraft, Jones2024physics, Dunik2018Directional, Gutierrez2024Class, Kulik2026nonlinearity, wang2023asymmetric} to filtering. Specifically, Wang and Forger \cite{wang2023asymmetric} recently proposed the asymmetric particle population density method (APPD) to simulate macroscopic behaviors of biological oscillators. In this framework, individual particles, treated as Gaussian components in GMFs, are propagated by solving ordinary differential equations (ODEs) reformulated from the underlying Fokker--Planck equation. The particles are then adaptively split, combined, and reweighted. Although the method appears to be a reasonable candidate for filtering applications, there remains room for further improvement. Its splitting algorithm is based on eigen-decomposition, which divides a particle along directions with the largest variance. However, this splitting strategy does not guarantee that the resulting particles satisfy the local linear approximation along a target direction, which is required for accurate state propagation of individual particles \cite{faubel2009split, Faubel2010Further, Horwood2011Adaptive, wang2023asymmetric, Demars2013entropy, Yu2014Adaptive, berchet2021adaptive}. Furthermore, although the combining algorithm based on the proximity of particle means is computationally efficient, it can distort the resulting mixture distribution due to differences in particle covariances \cite{wang2023asymmetric, runnalls2007kullback}.

Here, we propose a filtering-oriented split--combine scheme that dynamically adjusts the number and weights of particles. Our adaptive scheme theoretically guarantees a halving of the variance along a target direction while avoiding computationally demanding online numerical optimization and matrix inversion. Furthermore, it performs particle combining by assessing the distributional discrepancy between a candidate combined particle and the corresponding mixture of particles using a computationally efficient upper bound on the Kullback--Leibler (KL) divergence \cite{runnalls2007kullback}. Using this split--combine scheme, we developed an adaptive split–combine Gaussian mixture filter (AMF) for state estimation of nonlinear systems with asymmetric and multimodal state distributions. We compared AMF with several baselines, including recently developed Kalman-type filters and GMFs, across benchmark problems consisting of single and coupled slow--fast Van der Pol oscillators and the Lorenz attractor. AMF consistently outperformed the baselines. Importantly, AMF outperformed the original APPD-based GMF while using fewer particles during filtering. These results indicate that our adaptive split--combine scheme plays a critical role in nonlinear filtering. Lastly, we proposed a parallelized implementation of AMF and showed that, due to its suitability for parallel computation, it dramatically reduces the computational cost of highly nonlinear state-space estimation.

The remainder of this paper is organized as follows. Sec.~\ref{sec:Method} introduces AMF, with the problem formulation in Sec.~\ref{sec:probf} and an overview in Sec.~\ref{sec:overview}, followed by a detailed description in Secs.~\ref{sec:timeupdate}--~\ref{sec:breview}. In Sec.~\ref{sec:Results}, we present numerical analyses that demonstrate improvement of AMF over baselines. Sec.~\ref{sec:Conclusion} discusses the implications of our findings and future research directions.

\section{Methods} \label{sec:Method}

\subsection{Problem Formulation} \label{sec:probf}

A continuous--discrete dynamic system consists of two equations: the process equation and the measurement equation. The process equation describes the continuous-time evolution of the system states in the form of an It\^o--type SDE:

\begin{equation} \label{eq:sde}
    d\mathbf{x}_t = \mathbf{v}(\mathbf{x}_t)dt + \sqrt{\mathbf{K}}d\mathbf{W}_t
\end{equation}

\noindent where $\mathbf{x}_t$ denotes the $d$--dimensional state of the system at time $t$, $\mathbf{v} : \mathbb{R}^d \rightarrow \mathbb{R}^d$ is a drift function, $\mathbf{W}_t$ is a $d$--dimensional standard Wiener process, and $\mathbf{K}$ is a $d \times d$ positive semi-definite process noise matrix where $\mathbf{K} = \sqrt{\mathbf{K}}\sqrt{\mathbf{K}}^T$. The system states are observed indirectly through noisy measurements at discrete times, and their relationship is modeled by the measurement equation:

\begin{equation} \label{eq:meas}
    \mathbf{z}_{t_k} = h(\mathbf{x}_{t_{k}}) + \mathbf{r}
\end{equation}

\noindent where $\mathbf{z}_{t_{k}}$ is the $n$--dimensional measurements at discrete time instances $t_k$ for $k=1, 2, \dots$, $h : \mathbb{R}^d \rightarrow \mathbb{R}^{n}$ is the measurement function, and $\mathbf{r}$ is assumed to be a zero-mean Gaussian white noise with a covariance matrix $\mathbf{R}$.

The purpose of filtering is to compute the posterior density $p(\mathbf{x}_{t_k}|\mathbf{z}_{t_{1:k}})$, where $\mathbf{z}_{t_{1:k}}=\{\mathbf{z}_{t_1}, \cdots, \mathbf{z}_{t_k}\}$. This can be calculated recursively in two steps: the time-update step and the measurement-update step. In the time-update step, the predicted state density $p(\mathbf{x}_{t_k}|\mathbf{z}_{t_{1:k-1}})$ is obtained by solving the process equation \eqref{eq:sde} with the initial condition $p(\mathbf{x}_{t_{k-1}}|\mathbf{z}_{t_{1:k-1}})$, which is the previous posterior density. In the measurement-update step, the propagated density is updated at each measurement instant using Bayes’ theorem with the measurement equation \eqref{eq:meas}: 

\begin{equation}
\underbrace{p(\mathbf{x}_{t_{k}}|\mathbf{z}_{t_{1:k}})}_{\text{Posterior}} = \frac{1}{c_k}\underbrace{p(\mathbf{x}_{t_{k}}|\mathbf{z}_{t_{1:k-1}})}_{\text{Prior}}\underbrace{p(\mathbf{z}_{t_{k}}|\mathbf{x}_{t_k})}_{\text{Likelihood}}
\end{equation}
where $c_k=p(\mathbf{z}_{t_{k}}|\mathbf{z}_{t_{1:k-1}})=\int_{\mathbb{R}^d} p(\mathbf{x}_{t_{k}}|\mathbf{z}_{t_{1:k-1}})p(\mathbf{z}_{t_{k}}|\mathbf{x}_{t_k}) d\mathbf{x}_{t_{k}}$.

Although a recursive approach for the continuous--discrete filtering problem can be described conceptually as above, exact solutions exist only for a few special cases, such as when the drift function and the measurement function are linear \cite{arasaratnam2010cubature}. Except for these restricted cases, approximate solutions must be computed using numerical methods. However, existing numerical filters often cannot provide either sufficient accuracy or acceptable computational cost when estimating asymmetric and multimodal state densities over time. To overcome this limitation, we developed AMF that can accurately and efficiently track the non-Gaussian posterior density through an adaptive particle propagation--split--combine scheme in parallel. The next subsection presents an overview of the proposed method.

\subsection{Overview of the Method} \label{sec:overview}
Highly nonlinear systems exhibit non-Gaussian state distributions, and the degree of this non-Gaussianity changes over time. For example, in limit-cycle systems with fast–slow manifolds, state PDFs become elongated and bent along the orbit in the fast manifold, but this non-Gaussianity diminishes when passing through the slow manifold. Therefore, detecting non-Gaussianity during state propagation is essential for nonlinear state-space estimation. This inspired us to introduce a particle propagation--split--combine scheme in the time-update step of AMF (Fig. \ref{fig:fig1}). This scheme evaluates the validity of the Gaussian assumption (i.e., the local linear approximation) at each infinitesimal time step and adaptively splits or combines Gaussian particles to approximate the non-Gaussian PDF when required.

\begin{figure}[!htbp]
    \centering
    \includegraphics[width=\linewidth]{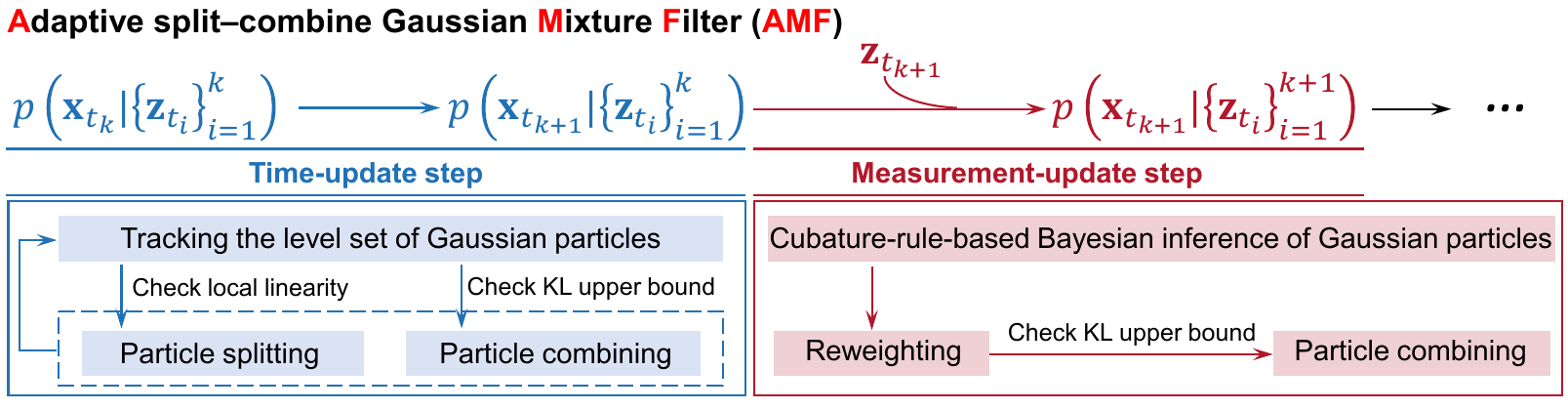}
    \caption{\textbf{Overview of the adaptive split–combine Gaussian mixture filter.} In the time-update step, single-particle propagation is performed by tracking its level set. Particles violating the local linear approximation are split at each infinitesimal time step without auxiliary online numerical optimization, guaranteeing halved variance along the target level-set-point direction. Particles are combined when the combined distribution is sufficiently close to their weighted sum in terms of the Kullback--Leibler divergence. In the measurement-update step, the mean, covariance, and weight of each particle are updated by numerical integration based on the cubature rule. See Secs.~\ref{sec:timeupdate} and~\ref{sec:measurementupdate} for details.}
    \label{fig:fig1}
\end{figure}

Specifically, AMF approximates a complex-shaped PDF with a mixture of Gaussian particles. The propagation of each particle is determined by solving ODEs along its level set, under a local linear approximation of the Fokker–Planck equation \cite{wang2021level}. During propagation, the quality of the local linear approximation is continuously evaluated without computing a Laplacian, which helps to avoid numerical stability issues in nonlinear state estimation \cite{wang2023asymmetric}. If the quality is not satisfactory, the particle is split into three. Notably, the splitting is performed without solving auxiliary numerical optimization problems and guarantees a halving of the variance along the target level-set-point direction. It thereby accelerates satisfaction of the local linear approximation assumption for each particle without requiring excessive splitting, which can accumulate numerical errors. Particles are also combined into a single Gaussian by evaluating information loss using an efficiently computable upper bound on the KL divergence \cite{runnalls2007kullback}. 

This scheme dramatically reduces the number of particles and thus the computational cost compared with APPD that inspired our approach, while improving accuracy (see below for details). Moreover, when implementing the propagation--split--combine scheme, the weight of each particle is automatically updated without any additional computation, unlike existing GMFs \cite{terejanu2008novel, Horwood2011Adaptive, Wang2020Adaptive, Li2021Efficient}.

Once the time-update step has been completed, the Gaussian particles and their mixture weights are updated using the latest measurement. First, each particle’s Gaussian density is corrected by propagating cubature points through the measurement function and assimilating them into the latest measurement, as in the CD-CKF measurement-update step \cite{arasaratnam2010cubature}. Next, each weight is updated by scaling it with the corresponding marginal likelihood approximated using a cubature rule. Finally, the updated particles undergo splitting and combining, as done in the time-update step.

To more efficiently implement AMF, we provide its parallel implementation. Note that this is possible because iterative operations, such as particle propagation and splitting, in AMF can be performed independently for each particle. In the following subsections, we provide detailed descriptions of the time-update step (Sec.~\ref{sec:timeupdate}), the measurement-update step (Sec.~\ref{sec:measurementupdate}), and the implementation details including parallel computing (Sec.~\ref{sec:imple_details}).

\subsection{Time-update step} \label{sec:timeupdate}
This section describes the time-update step of AMF, which consists of three parts: propagation of individual particles, splitting, and combining. These parts are presented in Sec.~\ref{sec:spp}, Sec.~\ref{sec:ps}, and Sec.~\ref{sec:pc}, respectively.

\subsubsection{Single-particle propagation} \label{sec:spp}
This section presents the algorithm for propagating a single Gaussian particle by solving the process equation \eqref{eq:sde} under the local linear approximation, as first proposed in \cite{wang2021level}. Unlike the CD-CKF, which directly solves the SDE numerically, this method converts it into the corresponding Fokker–Planck equation:

\begin{equation} \label{eq:fpe}
    \frac{\partial u}{\partial t} = \frac{1}{2}\nabla \cdot (\mathbf{K}\nabla u) - \nabla \cdot (\mathbf{v}u)
\end{equation}

\noindent where $u(\mathbf{x},t)$ denotes the time evolution of the state PDF. Next, the level-set equation of the Fokker--Planck equation is solved under the local linear approximation: $\mathbf{v(x)} \approx \mathbf{Jx}$ where $\mathbf{J}$ is the Jacobian matrix. Specifically, consider the level set $\mathcal{L}(t)$ of a single particle, initialized with the Gaussian distribution centered at zero without loss of generality:

\begin{equation} \label{eq:level_set}
    \mathcal{L}(t) := \{ \mathbf{x} \in \mathbb{R}^d \mid F(\mathbf{x},t):=\frac{u(\mathbf{x},t)}{u(\mathbf{0},t)} = c \},
\end{equation}

\noindent where $0 < c < 1$ is some fixed scalar constant and

\begin{equation} \label{eq:gaussian_initial}
    u(\mathbf{x},0) = \frac{1}{\sqrt{(2\pi)^d \det(\boldsymbol{\Sigma})}} \exp\left(-\frac{1}{2}\mathbf{x}^T \boldsymbol{\Sigma}^{-1} \mathbf{x}\right).
\end{equation}

Then, a velocity of the level set for $\mathcal{L}(0)$ is defined by a velocity field $\mathbf{v}_L$ satisfying the level-set equation:
\begin{equation} \label{eq:level_set_eq}
    \frac{dF}{dt}+ \mathbf{v}_L \cdot \nabla F =0.
\end{equation}
The velocity field $\mathbf{v}_L$ can be explicitly calculated under the local linear approximation (see Lemma 1 in \cite{wang2021level} for the detailed proof):
\begin{equation} \label{eq:level_set_velocity}
    \mathbf{v}_L \approx \mathbf{J}\mathbf{x} + \frac{1}{2}\mathbf{K}\boldsymbol{\Sigma}^{-1}\mathbf{x}.
\end{equation}

Using the velocity field $\mathbf{v}_L$ in \eqref{eq:level_set_velocity}, we trace the evolution of the column vectors of $\mathbf{M}$, where the covariance matrix $\mathbf{\Sigma}=\mathbf{M}\mathbf{M}^T$, as they lie on the same level set, and tracking this level set is equivalent to tracking the propagation of a Gaussian particle (see~\cref{app:app1} and \cite{wang2021level}). Specifically, let $\Sigma(0)=\mathbf{M}(0)\mathbf{M}(0)^T$ be the covariance matrix of an initial Gaussian density, and $\mathbf{M}(0)=\begin{bmatrix} \mathbf{x}_1(0) & \cdots & \mathbf{x}_d(0)\end{bmatrix}$, where $\mathbf{x}_i(0)$ is the $i$-th column vector of $\mathbf{M}(0)$. Then, $\mathbf{x}_i(0)$, representing a point on the level set, evolves according to the velocity field $\mathbf{v}_L$ in \eqref{eq:level_set_velocity}, which can be expressed as follows:

\begin{align} \label{eq:level_set_eq_x}
    \frac{d\mathbf{x}_i}{dt} &\approx \mathbf{J}\mathbf{x}_i + \frac{1}{2}\mathbf{K}\boldsymbol{\Sigma}^{-1}\mathbf{x}_i \approx \mathbf{v}(\bar{\mathbf{x}} + \mathbf{x}_i) -  \mathbf{v}_a(\bar{\mathbf{x}}, \mathbf{M}) + \frac{1}{2}\mathbf{K}\mathbf{M}^{-T}\mathbf{e}_i
\end{align}
where $\mathbf{e}_i$ is the unit vector with all entries equal to $0$ except for the $i$-th entry, $\bar{\mathbf{x}}$ is the center location of the particle, and  

\begin{align}
\mathbf{v}_a(\bar{\mathbf{x}}, \mathbf{M}) := \frac{1}{2d} \sum_{i=1}^{d} \left( \mathbf{v}(\bar{\mathbf{x}} + \mathbf{x}_i) + \mathbf{v}(\bar{\mathbf{x}} - \mathbf{x}_i) \right).
\end{align}
Note that the approximation in \eqref{eq:level_set_eq_x} allows us to avoid Jacobian computation. Thus, we adopt this form for simulating particle propagation. The matrix shorthand form of \eqref{eq:level_set_eq_x} is given by

\begin{equation} \label{eq:M_ode_avg}
    \frac{d\mathbf{M}}{dt} = \mathbf{v}(\bar{\mathbf{x}} + \mathbf{M}) - \mathbf{v}_a + \frac{1}{2}\mathbf{K}\mathbf{M}^{-T},
\end{equation}
and the averaged velocity of the center is given by 
\begin{equation} \label{eq:mean_vel}
    \frac{d\bar{\mathbf{x}}}{dt} = \mathbf{v}_a(\bar{\mathbf{x}}, \mathbf{M})
\end{equation}
where the matrix-vector additions are defined entrywise. Note that, in order to track the propagation of a single point $\mathbf{x}_i$ on the level set, the other points $\mathbf{x}_j$ as well as the center location $\bar{\mathbf{x}}$ of the Gaussian particle must also be considered. Hence, all points on the level set must be updated simultaneously by concatenating $\mathbf{M}$ and $\bar{\mathbf{x}}$ into a variable of $d \times (d+1)$ dimension, and solving it with any ODE solver. The single-particle update algorithm is illustrated in Algorithm \ref{alg:single_particle_update}. Note that individual Gaussian particles can be updated independently, which enables parallel computation and thus dramatically reduces computational time (see Sec.~\ref{sec:imple_details} for details).

\begin{algorithm}[hbt!]
\caption{Single-Particle Propagation}
\label{alg:single_particle_update}
\begin{algorithmic}[1]
\REQUIRE The center $\bar{\mathbf{x}}$ and the covariance square root $\mathbf{M}$ of the Gaussian particle at the previous time step.
\STATE Update $\bar{\mathbf{x}}$ and $\mathbf{M}$ by passing them as the state variables to an ODE solver with the derivative defined by \eqref{eq:M_ode_avg} and \eqref{eq:mean_vel}.
\RETURN The updated center and covariance square root
\end{algorithmic}
\end{algorithm}

\subsubsection{Particle splitting} \label{sec:ps}
The update of Gaussian particles relies on the local linear approximation of the drift term, as described above. However, this approximation becomes invalid when the size of a Gaussian particle grows as a result of continuous noise accumulation, or when the Gaussian structure is distorted by nonlinear dynamics, such as oscillations. To address this issue, we develop a particle splitting method inspired by \cite{wang2023asymmetric}. Specifically, we define the following relative error for the off-center points:

\begin{equation} \label{eq:relative_error}
    \varepsilon = \frac{\left\| \frac{1}{2}(\mathbf{v}(\bar{\mathbf{x}}+2\Delta\mathbf{x}^{\hat{\imath}}) - \mathbf{v}(\bar{\mathbf{x}})) - (\mathbf{v}(\bar{\mathbf{x}}+\Delta\mathbf{x}^{\hat{\imath}}) - \mathbf{v}(\bar{\mathbf{x}})) \right\|}{\left\| \mathbf{v}(\bar{\mathbf{x}}) \right\|}
\end{equation}
where $\bar{\mathbf{x}}$ is the center of the particle, and $\Delta\mathbf{x}^{\hat{\imath}}$ is the offset of the off-center point in direction $\hat{\imath}$. Using the relative error \eqref{eq:relative_error}, the soundness of the linear approximation is examined. Once the approximation is no longer valid, we need to reduce the particle covariance to shrink its effective support. This is achieved by splitting the particle into three thinner particles if the error exceeds a threshold  $\epsilon_{\text{threshold}}$ (see below for details). The selection of the threshold introduces a trade-off between computational cost and accuracy. Unless otherwise specified, $\epsilon_{\text{threshold}}$ is set to $0.05$, following \cite{wang2023asymmetric}.

Note that we use \eqref{eq:relative_error} instead of the criterion proposed in \cite{berchet2021adaptive} to avoid computing the Laplacian, thereby preserving the advantages of Laplacian-free approaches in single-particle propagation. This choice also improves numerical stability by eliminating Laplacian and Jacobian computations that may become unstable for near-singular covariance matrices. The numerator in \eqref{eq:relative_error} represents the sum of all higher-order ($\geq 2$) terms arising from the Taylor expansion of the drift function $\mathbf{v}$, scaled by a constant factor (see~\cref{app:app2}).

When determining whether particle splitting is required, we consider the errors computed by setting $\Delta\mathbf{x}^{\hat{\imath}}=\mathbf{x}_i$, the $i$-th column vector of $\mathbf{M}$. This allows us to assess the soundness of the approximation in \eqref{eq:level_set_eq_x}, which determines the numerical errors generated when solving \eqref{eq:M_ode_avg} and \eqref{eq:mean_vel}. Once the error exceeds a threshold, we split the particle into three thinner particles, each with half the variance in the direction of $\mathbf{x}_i$. Specifically, the covariance matrix $\boldsymbol{\Sigma}_{\text{split}}=\mathbf{N}\mathbf{N}^T$ of each child particle is computed by subtracting the appropriate orthogonal projections as follows:
\begin{align} \label{eq:new_cov_factor}
    \mathbf{N}_j &= \frac{1}{\sqrt{2}}\mathbf{x}_{j,\parallel_i} + \mathbf{x}_{j,\bot_i} \nonumber \\
    &= (\mathbf{x}_{j,\parallel_i} + \mathbf{x}_{j,\bot_i}) - \left(1-\frac{1}{\sqrt{2}}\right)\mathbf{x}_{j,\parallel_i} \nonumber \\
    &= \mathbf{x}_j - \left(1-\frac{1}{\sqrt{2}}\right) \frac{\langle \mathbf{x}_i, \mathbf{x}_j \rangle}{\langle \mathbf{x}_i, \mathbf{x}_i \rangle} \mathbf{x}_i, \quad 1 \le j \le d
\end{align}
where $\mathbf{N}_j$ is the $j$-th column of $\mathbf{N}$, $\mathbf{x}_{j,\parallel_i}$ and $\mathbf{x}_{j,\bot_i}$ denote the components of $\mathbf{x}_j$ parallel and perpendicular to $\mathbf{x}_i$, respectively, and $\langle \cdot, \cdot \rangle$ denotes the inner product. 

We then assign the central child particle a weight of $1-2w$, placing it at the same location as the parent particle. The two remaining child particles are assigned weights of $w$ each and are positioned at $\pm a$ relative to the center. With this configuration, we aim to find $w$ and $a$ such that the weighted sum of the child particles accurately represents the parent particle by solving the following optimization problem:
\begin{multline} \label{eq:full_split_opt}
    \min_{a,w}\max_{\mathbf{x}} \left| K(\mathbf{x} \mid \mathbf{0}, \boldsymbol{\Sigma}) \right. \\
    \left. - \left( (1-2w)K(\mathbf{x} \mid \mathbf{0}, \boldsymbol{\Sigma}_{\text{split}}) + wK(\mathbf{x} \mid -a\mathbf{u}_i, \boldsymbol{\Sigma}_{\text{split}}) + wK(\mathbf{x} \mid a\mathbf{u}_i, \boldsymbol{\Sigma}_{\text{split}}) \right) \right|
\end{multline}
where $K(\mathbf{x}\mid \boldsymbol{\mu}, \boldsymbol{\Sigma})$ denotes the PDF of the $d$--dimensional Gaussian distribution $\mathcal{N}(\boldsymbol{\mu},\boldsymbol{\Sigma})$, $\mathbf{0}\in\mathbb{R}^d$ is the zero vector, and $\mathbf{u}_i=\mathbf{x}_i/\|\mathbf{x}_i\|$ is the unit vector in the direction of $\mathbf{x}_i$. If $\mathbf{x}_i$ is aligned with any principal axis (i.e., an eigenvector) of the particle's covariance matrix, the multi-dimensional optimization problem \eqref{eq:full_split_opt} is reduced to $1$--dimensional problem that is independent of the remaining dimensions of the state space:
\begin{multline} \label{eq:split_opt} 
    \min_{a,w}\max_{x} \left| K(x \mid 0, \sigma^2) \right. \\
    \left. - \left( (1-2w)K(x \mid 0, \frac{\sigma^2}{2}) + wK(x \mid -a, \frac{\sigma^2}{2}) + wK(x \mid a, \frac{\sigma^2}{2}) \right) \right|
\end{multline}
where $K(x \mid \mu, \sigma^2)$ denotes the PDF of the $1$--dimensional Gaussian distribution with mean $\mu$ and variance $\sigma^2$. This problem can be solved numerically, yielding the optimal parameter values $a = 1.03332\sigma$ and $w = 0.21921$, as presented in \cite{wang2023asymmetric}.

However, this choice of $a$ and $w$ is typically suboptimal because the target level-set-point direction $\mathbf{x}_i$ is generally not aligned with any principal axis (i.e., an eigenvector) of the parent particle's covariance matrix. Due to this challenge, previous work splits the parent particle along a principal axis (rather than directly along $\mathbf{x}_i$) and adopts the values of $a$ and $w$ obtained by solving \eqref{eq:split_opt}, although there is no guarantee that this effectively reduces the linearization error \cite{faubel2009split, Faubel2010Further, Horwood2011Adaptive, wang2023asymmetric, Demars2013entropy, Yu2014Adaptive, berchet2021adaptive}. As a result, it may trigger unnecessary splitting and cause the number of particles to grow rapidly, leading to prohibitive computational costs and highly accumulated numerical errors in the optimization problem.

To tackle this limitation, we develop a particle splitting method that guarantees a reduction in variance along a vector $\mathbf{h}$ from the particle mean to an arbitrary level-set point, while still using the optimal parameter values derived from \eqref{eq:split_opt}. The key idea is to transform the Gaussian distribution into a new coordinate system where $\mathbf{h}$ is statistically uncorrelated with the other directions, allowing the optimal parameter values in \eqref{eq:split_opt} to be applied directly (Fig. \ref{fig:fig2}). After splitting, the resulting child distributions are transformed back to the original space. Specifically, without loss of generality, we consider a centered Gaussian distribution $\mathbf{x} \sim \mathcal{N}(\mathbf{0}, \mathbf{\Sigma})$ and aim to split it to reduce its variance along the direction $\mathbf{d} =\frac{\mathbf{h}}{||\mathbf{h}||}$.

\begin{figure}[t]
    \centering
    \includegraphics[width=\linewidth]{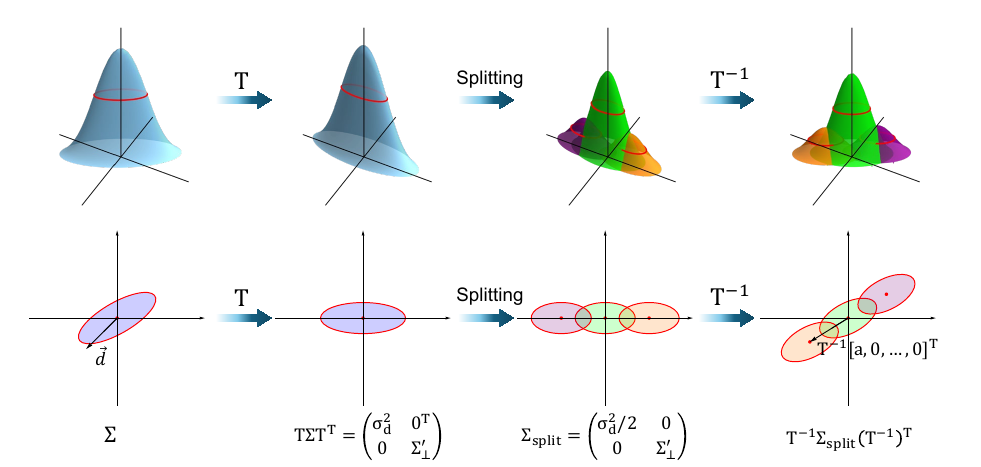}
    \caption{\textbf{Schematic diagram of the particle splitting method.} To reduce the variance along a specific direction $\mathbf{d}$, a linear transformation $\mathbf{T}$ is first applied to the particle so that the variance of the transformed particle along its principal axis equals the variance of the original particle along $\mathbf{d}$, i.e., $\sigma_d^2=\mathbf{d}^T\mathbf{\Sigma}\mathbf{d}$. The particle is then split into three child particles along the principal axis using the parameter values derived in \eqref{eq:split_opt}. The child particles are transformed back by applying $\mathbf{T}^{-1}$, yielding particles whose variance along $\mathbf{d}$ is halved. Note that the mean and covariance matrix of the child particles can be computed without explicitly computing either $\mathbf{T}$ or $\mathbf{T}^{-1}$  (see Sec.~\ref{sec:ps} for details).}
    \label{fig:fig2}
\end{figure}

As a first step, we construct a linear transformation $\mathbf{T}$ such that
\begin{equation}
\mathbf{T}=\begin{pmatrix} \mathbf{d}^T \\ \mathbf{B} \end{pmatrix}
\end{equation}
where the rows of $\mathbf{B}$ form an orthonormal basis for the null space of $(\mathbf{\Sigma}\mathbf{d})^T$. Such a matrix $\mathbf{T}$ can always be constructed, for example, using the Gram--Schmidt process. Note that $\mathbf{T}$ is invertible if the variance along $\mathbf{d}$ is nonzero (i.e., $\mathbf{d}^T \mathbf{\Sigma} \mathbf{d} \neq 0$), because this implies that $\mathbf{d}$ and $\mathbf{\Sigma}\mathbf{d}$ are not orthogonal, and thus $\mathbf{d}^T$ is linearly independent of the rows of $\mathbf{B}$. Importantly, the condition $\mathbf{d}^T \mathbf{\Sigma} \mathbf{d} \neq 0$ is satisfied when $\mathbf{d}$ is chosen as a level-set-point direction, because the variance along a level-set-point direction cannot be zero. Indeed, the condition is typically satisfied in filtering settings, since process noise is continuously introduced into the system. Hence, $\mathbf{T}$ can be considered as invertible.

Next, using $\mathbf{T}$, we transform the random vector $\mathbf{x}$ into $\mathbf{y} = \mathbf{T}\mathbf{x}$:
\begin{equation} \label{eq:y_partition}
\mathbf{y} = \begin{pmatrix} y_1 \\ \mathbf{y}_{\perp} \end{pmatrix} = \begin{pmatrix} \mathbf{d}^T \\ \mathbf{B} \end{pmatrix}\mathbf{x} = \begin{pmatrix} \mathbf{d}^T\mathbf{x} \\ \mathbf{B}\mathbf{x} \end{pmatrix}
\end{equation}
The first component, $y_1 = \mathbf{d}^T\mathbf{x}$, is a scalar random variable. Since it is a linear transformation of a Gaussian vector $\mathbf{x}$, $y_1$ itself follows a normal distribution with a mean of $E[y_1] = \mathbf{d}^T E[\mathbf{x}] = 0$. Its variance $\sigma_d^2$ is given by
\begin{equation} \label{eq:var_y1}
\sigma_d^2 = \text{Var}(y_1) = E[y_1y_1^T] = E[\mathbf{d}^T\mathbf{x}\mathbf{x}^T\mathbf{d}] = \mathbf{d}^T\mathbf{\Sigma}\mathbf{d}.
\end{equation}
This demonstrates that the variance of $y_1$ is equal to the variance of $\mathbf{x}$ along the direction $\mathbf{d}$, i.e., $\sigma_d^2 = \mathbf{d}^T\mathbf{\Sigma}\mathbf{d}$. Then, the covariance matrix of $\mathbf{y}$, $\mathbf{\Sigma}_\mathbf{y} = \mathbf{T}\mathbf{\Sigma}\mathbf{T}^T$, is computed as
\begin{equation} \label{eq:T_sigma_T_transpose}
\mat{\Sigma}_{\mat{y}}=\mat{T}\mat{\Sigma}\mat{T}^T = \begin{pmatrix} \vec{d}^T \\ \mat{B} \end{pmatrix} \mat{\Sigma} \begin{pmatrix} \vec{d} & \mat{B}^T \end{pmatrix} = \begin{pmatrix} \vec{d}^T\mat{\Sigma}\vec{d} & \vec{d}^T\mat{\Sigma}\mat{B}^T \\ \mat{B}\mat{\Sigma}\vec{d} & \mat{B}\mat{\Sigma}\mat{B}^T \end{pmatrix}.
\end{equation}
Note that, by construction of $\mat{B}$, the off-diagonal block $\mat{B}\mat{\Sigma}\vec{d}$ vanishes. Hence, $\mat{\Sigma}_{\mathbf{y}}$ is block-diagonal:
\begin{equation} \label{eq:sigma_prime_block}
\mat{\Sigma}_{\mathbf{y}} = 
\begin{pmatrix}
\sigma_d^2 & \vec{0}^T \\
\vec{0} & \mat{\Sigma}_{\mathbf{y},\perp}
\end{pmatrix}
\end{equation}
where $\mat{\Sigma}_{\mathbf{y},\perp}=\mat{B}\mat{\Sigma}\mat{B}^T$. 

The transformed random vector $\mathbf{y}$ now has a principal axis along which the variance equals that of $\mathbf{x}$ along $\mathbf{d}$ (i.e., $\mat{d}^T\mat{\Sigma}\mat{d}$). This enables us to split $\mathbf{y}$ into three child particles $\mathbf{y}_{\text{center}}$, $\mathbf{y}_{\text{left}}$, and $\mathbf{y}_{\text{right}}$ by directly adopting the values derived from \eqref{eq:split_opt}. The three child particles of $\mathbf{y}$ have the covariance matrix where the variance of the first component is halved:
\begin{equation} \label{eq:sigma_split_def}
\mat{\Sigma}_{\mat{y},\text{split}} = 
\begin{pmatrix}
\sigma_d^2/2 & \vec{0}^T \\
\vec{0} & \mat{\Sigma}_{\mat{y},\perp}
\end{pmatrix}
\end{equation}
Their means are given by
\begin{equation}
\boldsymbol{\mu}_{\mathbf{y},\text{center}} = 0, \hspace{4pt}
\boldsymbol{\mu}_{\mathbf{y},\text{left}} = [-a, 0, \dots, 0]^T, \hspace{4pt}
\boldsymbol{\mu}_{\mathbf{y},\text{right}} = [a, 0, \dots, 0]^T
\end{equation}
respectively, where $\boldsymbol{\mu}_{\mathbf{y},\text{center}}$, $\boldsymbol{\mu}_{\mathbf{y},\text{left}}$, and $\boldsymbol{\mu}_{\mathbf{y},\text{right}}$ denote the means of the center, left, and right child particles of $\mathbf{y}$, respectively, and $a = 1.03332\sigma_d$ is the displacement scalar.

Lastly, we apply the inverse transformation to $\mathbf{y}_{\text{center}}$, $\mathbf{y}_{\text{left}}$, and $\mathbf{y}_{\text{right}}$, obtaining the child particles of $\mathbf{x}$:
\begin{equation} \label{eq:final_particles}
\begin{aligned}
\mat{x}_{\text{center}}&=\mathbf{T}^{-1}\mathbf{y}_{\text{center}} \sim \mathcal{N}(0, \mat{\Sigma}_{\mat{x},\text{split}}) \\
\mat{x}_{\text{left}}&=\mathbf{T}^{-1}\mathbf{y}_{\text{left}} \sim \mathcal{N}(\mat{T}^{-1}\boldsymbol{\mu}_{\mathbf{y},\text{left}}, \mat{\Sigma}_{\mat{x},\text{split}}) \\
\mat{x}_{\text{right}}&=\mathbf{T}^{-1}\mathbf{y}_{\text{right}} \sim \mathcal{N}(\mat{T}^{-1}\boldsymbol{\mu}_{\mathbf{y},\text{right}}, \mat{\Sigma}_{\mat{x},\text{split}})
\end{aligned}
\end{equation}
where $\mat{\Sigma}_{\mat{x},\text{split}}=\mat{T}^{-1}\mat{\Sigma}_{\mat{y},\text{split}}(\mat{T}^{-1})^T$. Note that the $\mat{T}^{-1}\boldsymbol{\mu}_{\mathbf{y},\text{left}}$ and $\mat{T}^{-1}\boldsymbol{\mu}_{\mathbf{y},\text{right}}$ (i.e., $\mat{T}^{-1}[\pm a, 0, \dots, 0]^T$) can be expressed in a simpler form. Specifically, we have $\mat{T}^{-1}\mat{\Sigma}_\mat{y} = \mat{\Sigma}\mat{T}^T$ because $\mat{\Sigma}_{\mat{y}} = \mat{T}\mat{\Sigma}\mat{T}^T$. This can be written as follows:
\begin{equation}
    \mat{T}^{-1}\mat{\Sigma}_{\mat{y}}
    =
    \mat{T}^{-1}
    \begin{pmatrix}
    \sigma_d^2 & \vec{0}^T \\
    \vec{0} & \mat{\Sigma}_{\mat{y},\perp}
    \end{pmatrix}
    =\mat{\Sigma}
    \begin{pmatrix}
    \vec{d} & \mat{B}^{T}
    \end{pmatrix}
\end{equation}
By comparing the first columns on both sides of the second equality, we can observe that the vector obtained by multiplying the first column of $\mat{T}^{-1}$ by $\sigma_d^2$ is equal to $\mat{\Sigma}\vec{d}$. Therefore, the following relationship holds:
\begin{equation}  \label{eq:displacement}
\mat{T}^{-1}[a, 0, \dots, 0]^T=a \frac{\Sigma\vec{d}}{\sigma_d^{2}}=1.03332\sigma_d\frac{\Sigma\vec{d}}{\sigma_d^{2}}=\frac{1.03332}{\sigma_d}\Sigma\vec{d}
\end{equation}
This shows that it is not necessary to explicitly compute either the transformation matrix $\mat{T}$ or its inverse when calculating the means $\mat{T}^{-1}\boldsymbol{\mu}_{\mathbf{y},\text{left}}$ and $\mat{T}^{-1}\boldsymbol{\mu}_{\mathbf{y},\text{right}}$ of the child particles in \eqref{eq:final_particles}. We will later show that computing either $\mat{T}$ or $\mat{T}^{-1}$ is also unnecessary for evaluating their covariance matrix $\mat{\Sigma}_{\mat{x},\text{split}}$.

Next, we show that the splitting method strictly reduces the variance of the particle distribution along a given direction $\vec{d}$. Specifically, we decompose $\vec{x} \sim \mathcal{N}(\boldsymbol{0}, \mat{\Sigma})$ into two components: $\vec{x}_{\parallel}$ that is parallel to the covariance vector between $\mathbf{x}$ and its projection onto direction $\vec{d}$, i.e., $\vec{d}^T \vec{x}$, and $\vec{x}_{\perp}$ that is uncorrelated with $\vec{d}^T \vec{x}$. That is,
\[
\vec{x} = \vec{x}_{\parallel} + \vec{x}_{\perp} \quad \text{where} \quad \vec{x}_{\parallel}  \parallel \text{Cov}(\mathbf{x},\vec{d}^T \vec{x})=\mathbf{\Sigma}\mathbf{d} \quad \text{and} \quad \text{Cov}(\vec{x}_{\perp}, \vec{d}^T \vec{x}) = 0.
\]
This decomposition can always be constructed by taking
\begin{equation} \label{eq:x_parallel_perp_def}
\vec{x}_{\parallel} := \frac{\vec{d}^T \vec{x}}{\vec{d}^T\mat{\Sigma}\vec{d}}\mat{\Sigma}\vec{d}, \quad 
\vec{x}_{\perp} := \vec{x} - \vec{x}_{\parallel},
\end{equation}
because $\vec{x}_{\parallel}$ is a scalar multiple of  $\mathrm{Cov}(\vec{x}, \vec{d}^T \vec{x}) = \mat{\Sigma}\vec{d}$ with scalar $\frac{\vec{d}^T \vec{x}}{\vec{d}^T\mat{\Sigma}\vec{d}}$ and
\begin{align}
    \mathrm{Cov}(\vec{x}_{\perp}, \vec{d}^T \vec{x})
    &= \mathbb{E}\big[\vec{x}_{\perp} \vec{d}^T \vec{x}\big] \nonumber\\
    &= \mathbb{E}\big[(\vec{x}-\vec{x}_{\parallel})\vec{d}^T \vec{x}\big] \nonumber\\
    &= \mathbb{E}[\vec{x}\vec{d}^T \vec{x}] - \mathbb{E}\left[\frac{\vec{d}^T \vec{x}\,\vec{d}^T \vec{x}}{\vec{d}^T\mat{\Sigma}\vec{d}}\,\mat{\Sigma}\vec{d}\right] \nonumber\\
    &= \mat{\Sigma}\vec{d} - \frac{\vec{d}^T\mat{\Sigma}\vec{d}}{\vec{d}^T\mat{\Sigma}\vec{d}} \mat{\Sigma}\vec{d}=\mat{\Sigma}\vec{d} - \mat{\Sigma}\vec{d} = \vec{0}.
    \label{eq:cov_xperp_y}
\end{align}
Note that $\mathrm{Cov}(\vec{x}_{\parallel}, \vec{d}^T \vec{x})=\mathbf{\Sigma}\vec{d}=\mathrm{Cov}(\vec{x}, \vec{d}^T \vec{x})$:
\begin{align}
    \mathrm{Cov}(\vec{x}_{\parallel}, \vec{d}^T \vec{x})
    &= \mathbb{E}\big[\vec{x}_{\parallel} \vec{d}^T \vec{x}\big] \nonumber\\
    &= \mathbb{E}\left[\frac{\vec{d}^T \vec{x}\,\vec{d}^T \vec{x}}{\vec{d}^T\mat{\Sigma}\vec{d}}\,\mat{\Sigma}\vec{d}\right] \nonumber\\
    &= \frac{\vec{d}^T\mat{\Sigma}\vec{d}}{\vec{d}^T\mat{\Sigma}\vec{d}} \mat{\Sigma}\vec{d} \nonumber\\
    &= \mat{\Sigma}\vec{d} = \text{Cov}(\vec{x},\vec{d}^T \vec{x}).
    \label{eq:cov_xparallel_y}
\end{align}
This implies that $\vec{x}_{\parallel}$ exactly captures the component of $\vec{x}$ that is correlated with $\vec{d}^T \vec{x}$, while $\vec{x}_{\perp}$ represents the residual component that is uncorrelated with $\vec{d}^T \vec{x}$. It also follows that $\vec{x}_{\perp}$ and $\vec{d}^T \vec{x}$ are independent from \eqref{eq:cov_xperp_y} since $\vec{x}$ is Gaussian and $(\vec{x}_{\perp}, \vec{d}^T \vec{x})$ are jointly Gaussian. This results in $\text{Cov}(\vec{x}_{\perp},\vec{x}_{\parallel})=\vec{0}$ because $\vec{x}_{\parallel}$ is a linear function of $\vec{d}^T \vec{x}$ \eqref{eq:x_parallel_perp_def} and $\text{Cov}(\vec{x}_{\perp},\vec{d}^T \vec{x})=\vec{0}$ \eqref{eq:cov_xperp_y}. Therefore,
\[
\mat{\Sigma} = \mat{\Sigma}_{\parallel} + \mat{\Sigma}_{\perp}.
\]
where $\mat{\Sigma}_{\parallel}$ and $\mat{\Sigma}_{\perp}$ denote $\text{Var}(\vec{x}_{\parallel})$ and $\text{Var}(\vec{x}_{\perp})$, respectively, and 
\begin{align}
    \mat{\Sigma}_{\parallel}
    &= \text{Var}\left(\frac{\vec{d}^T \vec{x}}{\vec{d}^T\mat{\Sigma}\vec{d}}\mat{\Sigma}\vec{d}\right) \nonumber\\
    &= \frac{\text{Var}(\vec{d}^T \vec{x})}{(\vec{d}^T\mat{\Sigma}\vec{d})^2}(\mat{\Sigma}\vec{d})(\mat{\Sigma}\vec{d})^T \nonumber\\
    &= \frac{(\mat{\Sigma}\vec{d})(\mat{\Sigma}\vec{d})^T}{\vec{d}^T\mat{\Sigma}\vec{d}}.
    \label{eq:sigma_parallel_def}
\end{align}
This implies that $\mat{\Sigma}_{\parallel}$ represents the covariance contributed by the component of $\vec{x}$ aligned with $\text{Cov}(\vec{x},\vec{d}^T \vec{x})=\mat{\Sigma}\vec{d}$, i.e., the variability of $\vec{x}$ explained by its projection onto direction $\vec{d}$. This suggests that $\mat{\Sigma}_{\parallel}$ should be halved, while leaving the orthogonal component unchanged, if the proposed splitting method in \eqref{eq:final_particles} correctly reduces the variability of $\vec{x}$ explained by its projection onto $\vec{d}$. We confirm this in \Cref{thm:main_result}.

\begin{theorem} \label{thm:main_result}
The covariance $\mat{\Sigma}_{\mat{x},\text{split}}$ of the child particles in \eqref{eq:final_particles} is given by
\begin{equation} \label{eq:sigma_final_main_result}
\mat{\Sigma}_{\mat{x},\text{split}} = \mat{\Sigma} - \frac{1}{2}\mat{\Sigma}_{\parallel}.
\end{equation}
\end{theorem}
\begin{proof}
In the splitting procedure, the variance along the first coordinate of $\mathbf{y}=\mat{T}\mathbf{x}$ is reduced by half \eqref{eq:sigma_split_def}:
\begin{align}
\mat{\Sigma}_{\mathbf{y},\text{split}} 
= \mat{\Sigma}_{\mathbf{y}} - \frac{\sigma_d^2}{2}\vec{e}_1\vec{e}_1^\top,
\label{eq:y_split_eq}
\end{align}
where $\vec{e}_1$ denotes the first standard basis vector in $\mathbb{R}^d$ and $\sigma^2_{d}=\mat{d}^T\mat{\Sigma}\mat{d}$. Mapping \eqref{eq:y_split_eq} back to the $\mathbf{x}$-space via the inverse transformation, we obtain
\begin{align}
\mat{\Sigma}_{\mat{x},\text{split}} 
&= \mat{T}^{-1} \mat{\Sigma}_{\mathbf{y},\text{split}} (\mat{T}^{-1})^T \nonumber\\
&= \mat{T}^{-1} \left( \mat{\Sigma}_{\mathbf{y}} - \frac{\sigma_d^2}{2}\vec{e}_1\vec{e}_1^T \right) (\mat{T}^{-1})^T \nonumber\\
&= \mat{\Sigma} - \frac{\sigma_d^2}{2} (\mat{T}^{-1}\vec{e}_1)(\mat{T}^{-1}\vec{e}_1)^T.
\end{align}
Because $\mat{T}^{-1}\vec{e}_1 = \frac{\mat{\Sigma}\vec{d}}{\sigma_d^2}$ \eqref{eq:displacement}, we have

\begin{align}
\mat{\Sigma}_{\mat{x},\text{split}} 
&= \mat{\Sigma} - \frac{\sigma_d^2}{2}
\left( \frac{\mat{\Sigma}\vec{d}}{\sigma_d^2} \right)
\left( \frac{\mat{\Sigma}\vec{d}}{\sigma_d^2} \right)^T \nonumber\\
&= \mat{\Sigma} - \frac{1}{2} \frac{(\mat{\Sigma}\vec{d})(\mat{\Sigma}\vec{d})^T}{\vec{d}^T\mat{\Sigma}\vec{d}} 
\nonumber\\
&= \mat{\Sigma} - \frac{1}{2} \mat{\Sigma}_{\parallel}.
\label{eq:thm_split_result}
\end{align}
Note that the third equality follows from \eqref{eq:sigma_parallel_def}.
\end{proof}

\begin{remark}
\Cref{thm:main_result} provides a theoretical justification for the particle splitting method in \eqref{eq:final_particles}. In particular, it shows that the method selectively reduces the covariance captured by projecting $\mathbf{x}$ onto direction $\mathbf{d}$ and reconstructing it optimally as
\[
\operatorname{Var}\!\left(\mathbb{E}[\mathbf{x} \mid \mathbf{d}^T \mathbf{x}]\right)
=
\frac{(\Sigma \mathbf{d})(\Sigma \mathbf{d})^T}{\mathbf{d}^T \Sigma \mathbf{d}}.
\]
\end{remark}

\begin{remark}
The means and covariances of the child particles can be computed by 
\eqref{eq:displacement} and \eqref{eq:thm_split_result} without explicitly computing either $\mat{T}$ or its inverse. 
Therefore, the splitting procedure remains computationally tractable, despite involving a transformation and its inverse.
\end{remark}

Moreover, we show that the error introduced by particle splitting \eqref{eq:final_particles} is identical to the error obtained from the corresponding $1$--dimensional particle splitting optimization problem \eqref{eq:split_opt} (see~\cref{app:app3}). This justifies the use of the parameter values $a = 1.03332\sigma$ and $w = 0.21921$ from \eqref{eq:split_opt} to determine the displacements and weights of the child particles at every splitting event, without solving an auxiliary numerical optimization problem.

Taken together with the above remarks, this demonstrates that the proposed method enables well-tuned particle splitting. This is important from the perspectives of both accuracy and computational cost. Specifically, unnecessary splitting increases the number of Gaussian particles beyond what is needed, resulting in the excessive executions of Algorithm~\ref{alg:single_particle_update} and thus increasing the computational burden. Moreover, repeatedly solving the optimization problem \eqref{eq:full_split_opt} during splitting due to unnecessary splits may lead to a non-negligible accumulation of numerical errors, thereby affecting the filtering results.

The particle splitting algorithm is illustrated in Algorithm~\ref{alg:particle_splitting}. Note that when particle splitting is applied iteratively, the weights $W$ of some particles can become extremely small, causing them to contribute negligibly to the overall distribution. In such cases, particles with weights below a certain threshold $W_{\text{threshold}}$ are pruned. This threshold is set to $10^{-4}$ unless otherwise specified. The choice of threshold introduces a trade-off between computational cost and accuracy, as retaining more small-weight particles increases the number of executions of Algorithm~\ref{alg:single_particle_update}.

\begin{algorithm}[H]
\caption{Particle Splitting} 
\label{alg:particle_splitting}
\begin{algorithmic}[1]

\REQUIRE $W_{\text{threshold}}\in(0,1]$, $\epsilon_{\text{threshold}}>0$, a Gaussian particle $p = (W, \bar{\mathbf{x}}, \mathbf{\Sigma})$ where $W \in (0,1]$ is the weight, $\bar{\mathbf{x}} \in \mathbb{R}^d$ and $\mathbf{\Sigma}= \mathbf{M} {\mathbf{M}}^T \in \mathbb{R}^{d \times d}$ denote the mean and covariance, respectively.

\IF{$W < W_{\text{threshold}}$}
    \STATE \textbf{return} $\emptyset$
\ENDIF

\FOR{$i = 1$ to $d$}
    \STATE Calculate $\epsilon_i$ using \eqref{eq:relative_error} with 
    $\bar{\mathbf{x}}$ and $\Delta \mathbf{x}^{\hat{i}} = \mathbf{M}_i$ where $\mathbf{M}_i$ is $i$-th column of $\mathbf{M}$
\ENDFOR

\STATE $k \leftarrow \arg\max_{1 \le i \le d} \epsilon_i$

\IF{$\epsilon_k \le \epsilon_{\text{threshold}}$}
    \STATE \textbf{return} $\{(W, \bar{\mathbf{x}}, \mathbf{\Sigma})\}$
\ENDIF

\STATE $\mathbf{d}_k \leftarrow \mathbf{M}_k / \|\mathbf{M}_k\|$

\STATE Compute
\[
\boldsymbol{\mu}= \frac{a}{\mathbf{d}_k^T \mathbf{\Sigma} \mathbf{d}_k}\mathbf{\Sigma}\mathbf{d}_k,\,\,\, \mathbf{\Sigma}_{\parallel,k}
=
\frac{
(\mathbf{\Sigma}\mathbf{d}_k)(\mathbf{\Sigma}\mathbf{d}_k)^T
}{
\mathbf{d}_k^T \mathbf{\Sigma} \mathbf{d}_k
}
\]
where $a=1.03332\sqrt{\mathbf{d}_k^T \mathbf{\Sigma} \mathbf{d}_k}$ and $w=0.21921$ are the values from \eqref{eq:split_opt}.

\STATE Generate three child particles: center, left, and right, denoted by $p_{\text{center}}$, $p_{\text{left}}$, and $p_{\text{right}}$:
\[
p_\text{center} =
\left((1-2w)W, 
\bar{\mathbf{x}},
\mathbf{\Sigma} - \tfrac{1}{2}\mathbf{\Sigma}_{\parallel,k}
\right)
\]
\[
p_\text{left} =
\left(wW, 
\bar{\mathbf{x}} - \boldsymbol{\mu},
\mathbf{\Sigma} - \tfrac{1}{2}\mathbf{\Sigma}_{\parallel,k}
\right),\,\,\,p_\text{right} =
\left(wW, 
\bar{\mathbf{x}} + \boldsymbol{\mu},
\mathbf{\Sigma} - \tfrac{1}{2}\mathbf{\Sigma}_{\parallel,k}
\right)
\]

\STATE \textbf{return}
$p_{\text{center}}, p_{\text{left}}, p_\text{right}$

\end{algorithmic}
\end{algorithm}

\subsubsection{Particle combining} \label{sec:pc}
Particle splitting generates new particles. Thus, using only particle splitting for filtering leads to a geometric increase in the number of particles over time, making the method impractical. To address this issue, particle combining is required. This need has also been recognized in previous work \cite{zhang2022gaussian, faubel2009split, wang2023asymmetric, runnalls2007kullback, Yu2014Adaptive}. For example, Wang and Forger \cite{wang2023asymmetric} recently developed a method to determine which particles should be combined by partitioning the state space into cubic-shaped grid cells and considering the proximity of particle means in the same cell. Specifically, if the mean distance among particles, $\{\vec{x}_i \sim \mathcal{N}(\boldsymbol{\mu}_i, \mat{\Sigma}_i)\}_{i=1}^{n}$, within the same cell is less than a certain radius, they are combined into a single Gaussian particle, $\vec{x}_{\text{com}} \sim \mathcal{N}(\boldsymbol{\mu}_{\text{com}}, \mat{\Sigma}_{\text{com}})$:
\begin{equation}
\label{eq:combined_mean_cov}
\begin{aligned}
\boldsymbol{\mu}_{\text{com}}
&=
\frac{1}{w_{\text{com}}}
\sum_{i=1}^{n} w_i \boldsymbol{\mu}_i,\\
\boldsymbol{\Sigma}_{\text{com}}
&=
\underbrace{\frac{1}{w_{\text{com}}}\sum_{i=1}^{n} w_i \mathbf{\Sigma}_i}_{\mathbf{\Sigma}_{\text{W}}}
+
\underbrace{\frac{1}{w_{\text{com}}}\sum_{i=1}^{n} w_i
(\boldsymbol{\mu}_i-\boldsymbol{\mu}_{\text{com}})
(\boldsymbol{\mu}_i-\boldsymbol{\mu}_{\text{com}})^T}_{\mathbf{\Sigma}_{\text{B}}}
\end{aligned}
\end{equation}
where $w_i$ is the weight of $\vec{x}_i$, $w_{\text{com}}=\sum_{i=1}^n w_i$ is the weight of the combined particle, $\boldsymbol{\Sigma}_{\text{W}}$ represents the within-component contribution, and $\boldsymbol{\Sigma}_{\text{B}}$ represents the between-component contribution. Note that the equality of the variances is a direct result of the law of total variance.

The method is computationally efficient because it only uses the proximity of particle means as the combining criterion. However, it may make incorrect combining decisions in certain cases. For instance, if particles have similar means but significantly different covariances, the distribution of the combined particle can differ substantially from that of their weighted sum (Fig. \ref{fig:fig3}(a)). Furthermore, even when the means are far apart, sufficiently large covariances can make the distribution of the combined particle similar to that of their weighted sum (Fig. \ref{fig:fig3}(b)). In such cases, incorrect particle-combining decisions based on particle means can lead to non-negligible numerical errors.

\begin{figure}[!htbp]
    \centering
    \includegraphics[width=\linewidth]{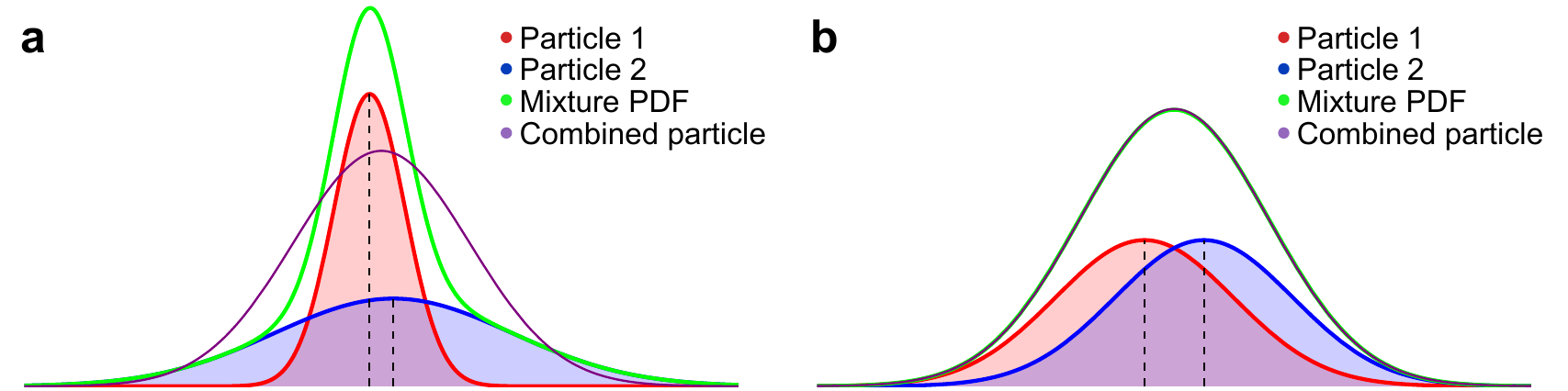}
    \caption{\textbf{Particle-combining decisions based solely on particle means can be inaccurate, leading to numerical errors and unnecessary computational cost.} \textbf{(a)} When the means of particles 1 (red) and 2 (blue) are close but their variances differ substantially, the PDF of the combined particle (purple) can differ significantly from the weighted mixture PDF (green). Thus, combining particles solely because their means are close can lead to numerical errors. \textbf{(b)} When the means of two particles are far apart but their variances are sufficiently large, the weighted mixture PDF can still be approximated well by the PDF of the combined particle. Thus, relying solely on particle means can result in an unnecessarily large number of particles, increasing computational cost.}
    \label{fig:fig3}
\end{figure}

To prevent this, we focus on the similarity between the distributions. Specifically, let us consider the Gaussian mixture
\begin{equation} \label{eq:gmm_reduction}
\begin{split}
    \mathbf{x} = \sum_{i=1}^{N} w_i \mathbf{x}_i, \quad \mathbf{x}_i \sim \mathcal{N}(x; \boldsymbol{\mu}_i, \mathbf{\Sigma}_i), \quad \sum_{i=1}^{N} w_i = 1.
\end{split}
\end{equation}
In the mixture, we aim to determine whether two particles, denoted without loss of generality as $\mathbf{x}_1$ and $\mathbf{x}_2$, should be combined to $\mathbf{x}_{\text{com}}$ computed using \eqref{eq:combined_mean_cov}. For this, we first consider the KL divergence between the Gaussian mixtures before and after combining:
\begin{multline} \label{eq:kl_div}
    D_{\text{KL}}(w_1\pi_1 + w_2\pi_2 + \pi_{\text{res}}, \pi_{\text{com}} + \pi_{\text{res}}) = \\
    \int_{\mathbb{R}^d} (w_1\pi_1 + w_2\pi_2 + \pi_{\text{res}})(x) \log \frac{(w_1\pi_1 + w_2\pi_2 + \pi_{\text{res}})(x)}{(\pi_{\text{com}} + \pi_{\text{res}})(x)} dx.
\end{multline}
where $\pi_i$ and $\pi_{\text{res}}$ denote the PDFs of $\mathbf{x}_i$ and the weighted sum of residual particles $\sum_{i=3}^{N} w_i \mathbf{x}_i$, respectively. Note that, in \eqref{eq:kl_div}, $\pi_{\text{res}}$ is included because we focus on the effect of particle combining on the entire mixture $\mathbf{x}$ in \eqref{eq:gmm_reduction}, rather than only on the sub-mixture consisting of $\mathbf{x}_1$ and $\mathbf{x}_2$. Although the KL divergence \eqref{eq:kl_div} provides a principled criterion for combining particles, its computation is intractable during filtering, particularly when a large number of particle pairs must be considered. Thus, we instead utilize the computable upper bound $\mathcal{B}(\mathbf{x}_1,\mathbf{x}_2)$ on the KL divergence inspired by \cite{runnalls2007kullback} and \cite{wills2023bayesianfilteringalgorithmgaussian}:
\begin{equation} \label{eq:kl_bound1}
    D_{\text{KL}}(w_1\pi_1 + w_2\pi_2 + \pi_{\text{res}}, \pi_{\text{com}} + \pi_{\text{res}}) \le \mathcal{B}(\mathbf{x}_1,\mathbf{x}_2)
\end{equation}
where
\begin{equation} \label{eq:kl_bound_final}
    \mathcal{B}(\mathbf{x}_1,\mathbf{x}_2) = \frac{1}{2}\left[(w_1+w_2)\log \det(\mathbf{\Sigma}_{\text{com}}) - w_1\log \det(\mathbf{\Sigma}_1) - w_2\log \det(\mathbf{\Sigma}_2)\right].
\end{equation}

\vspace{1.3ex}

\noindent This upper bound quantifies the information loss caused by replacing two Gaussian particles with a single Gaussian particle in the mixture. Although $\mathcal{B}(\mathbf{x}_i,\mathbf{x}_j)$ is expressed only in terms of the covariance determinants, it is an upper bound on the KL divergence between the Gaussian mixtures before and after combining, rather than a direct measure of the covariance change itself \cite{runnalls2007kullback}. The detailed step-by-step derivation of \eqref{eq:kl_bound1} and \eqref{eq:kl_bound_final} is provided in~\cref{app:app4} for completeness, since our notation and context differ from those in \cite{runnalls2007kullback}. Note that computing \eqref{eq:kl_bound_final} is computationally efficient as it only involves determinant evaluations and avoids expensive procedures such as Monte Carlo sampling for KL divergence estimation.

Using \eqref{eq:kl_bound_final}, we perform particle combining (Fig. \ref{fig:fig4}). Specifically, we first partition the state space into a grid, as done in \cite{wang2023asymmetric}. Within each grid cell, we compute the minimum value of $\mathcal{B}(\mathbf{x}_i,\mathbf{x}_j)$ over all particle pairs. If the minimum value falls below a threshold $C_{\text{threshold}}$, the corresponding particle pair is combined, as illustrated in Algorithm~\ref{alg:kl_gmm_reduction}. $C_{\text{threshold}}$ is set to 0.05 unless otherwise specified. Note that the combining process is iterated in each grid cell until the minimum value is greater than $C_{\text{threshold}}$.

\begin{figure}[h]
    \centering
    \includegraphics[width=0.8\linewidth]{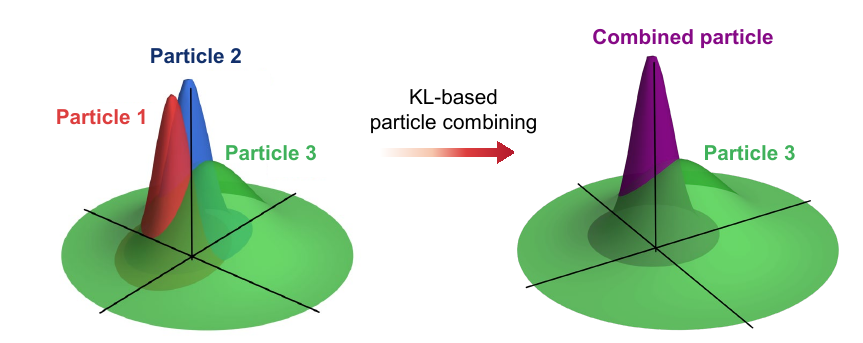}
    \caption{\textbf{Particle combining based on the Kullback--Leibler divergence upper bound.} The upper bound in \eqref{eq:kl_bound_final} is evaluated for each particle pair. If the minimum value among all pairs is below a threshold, the corresponding pair is combined, as illustrated for particles 1 and 2.}
    \label{fig:fig4}
\end{figure}

\begin{algorithm}[h]
\caption{Particle Combining}
\label{alg:kl_gmm_reduction}
\begin{algorithmic}[1]

\REQUIRE $C_{\text{threshold}} > 0$, particle set and corresponding weights
$\{\vec{x}_i,w_i\}_{i \in S}$ in a grid cell ($S$: particle indices in the cell)

\STATE Compute $\mathcal{B}(\mathbf{x}_i,\mathbf{x}_j)$ using \eqref{eq:kl_bound_final} for all $i,j \in S$ with $i < j$
\STATE Find
\[
(i^\star, j^\star)=\arg\min_{i,j \in S,\; i < j}\mathcal{B}(\mathbf{x}_i,\mathbf{x}_j)
\]

\IF{$\mathcal{B}(\mathbf{x}_{i^\star},\mathbf{x}_{j^\star}) < C_{\text{threshold}}$}
    \STATE Compute the combined particle and its weight $\{\mathbf{x}_{\mathrm{com}}, w_{\mathrm{com}}\}$ from $\mathbf{x}_{i^\star}$ and $\mathbf{x}_{j^\star}$ using \eqref{eq:combined_mean_cov}:
    \STATE \hspace{1em}
    $\{\mathbf{x}_{i^\star},w_{i^\star}\}
    \leftarrow
    \{\mathbf{x}_{\text{com}},w_{\text{com}}\}$
    \STATE Remove $j^\star$ from $S$: $S \leftarrow S \setminus \{j^\star\}$
\ENDIF

\RETURN Updated particle set and corresponding weights $\{\vec{x}_i,w_i\}_{i \in S}$

\end{algorithmic}
\end{algorithm}

\subsection{Measurement-update step} \label{sec:measurementupdate}
As a result of the time-update step, we obtain the predicted (i.e., prior) state distribution, represented as a Gaussian mixture:
\begin{equation}
    p(\mathbf{x}_{t_k}\mid \mathbf{z}_{t_{1:k-1}})
    =
    \sum_{i=1}^{M} \hat{w}_i \hat{\pi}_i(\mathbf{x}_{t_k})
\label{eq:pred_state_dis}
\end{equation}
where $\mathbf{z}_{t_{1:k-1}}=\{\mathbf{z}_{t_1}, \cdots, \mathbf{z}_{t_{k-1}}\}$, $\hat{w}_i$ is the predicted weight of the $i$-th Gaussian component, and $\hat{\pi}_i$ denotes its corresponding PDF with mean $\hat{\boldsymbol{\mu}}_i$ and covariance matrix $\hat{\mathbf{\Sigma}}_i$. The predicted state distribution \eqref{eq:pred_state_dis} is updated using Bayes' theorem when a new measurement $\mathbf{z}_{t_k}$ is given:
\begin{align}
    p(\mathbf{x}_{t_k} \mid \mathbf{z}_{t_{1:k}})
    &=
    \frac{
        p(\mathbf{z}_{t_k} \mid \mathbf{x}_{t_k})
        \sum_{i=1}^{M} \hat{w}_i \hat{\pi}_i(\mathbf{x}_{t_k})
    }{
        \int
        p(\mathbf{z}_{t_k} \mid \mathbf{x}_{t_k})
        \sum_{i=1}^{M} \hat{w}_i \hat{\pi}_i(\mathbf{x}_{t_k})
        \, d\mathbf{x}_{t_k}
    }\\[1.4ex]
    &=
    \frac{
        \sum_{i=1}^{M} \hat{w}_i p(\mathbf{z}_{t_k} \mid \mathbf{x}_{t_k}) \hat{\pi}_i(\mathbf{x}_{t_k})
    }{
        \sum_{i=1}^{M} \hat{w}_i \int
        p(\mathbf{z}_{t_k} \mid \mathbf{x}_{t_k})\hat{\pi}_i(\mathbf{x}_{t_k})
        \, d\mathbf{x}_{t_k}
    }\\[1.4ex]
    &=\frac{1}{c_k}\sum_{i=1}^{M} \hat{w}_i p(\mathbf{z}_{t_k} \mid \mathbf{x}_{t_k})\hat{\pi}_i(\mathbf{x}_{t_k})
\end{align}
where $c_k = \sum_{i=1}^{M} \hat{w}_i \int p(\mathbf{z}_{t_k} \mid \mathbf{x}_{t_k})\hat{\pi}_i(\mathbf{x}_{t_k}) \, d\mathbf{x}_{t_k}$. This expression can be further rewritten as
\begin{align}
    p(\mathbf{x}_{t_k} \mid \mathbf{z}_{t_{1:k}})
    &=
    \sum_{i=1}^{M}
    \Biggl[
    \frac{\hat{w}_i}{c_k}
    \left(
        \int
        p(\mathbf{z}_{t_k}\mid\mathbf{x}_{t_k})
        \hat{\pi}_i(\mathbf{x}_{t_k})
        d\mathbf{x}_{t_k}
    \right)
    \nonumber \\[1ex]
    &\qquad \cdot
    \left(
        \frac{
            p(\mathbf{z}_{t_k}\mid\mathbf{x}_{t_k})
            \hat{\pi}_i(\mathbf{x}_{t_k})
        }{
            \int
            p(\mathbf{z}_{t_k}\mid\mathbf{x}_{t_k})
            \hat{\pi}_i(\mathbf{x}_{t_k})
            d\mathbf{x}_{t_k}
        }
    \right)
    \Biggr].
\end{align}
This implies that the updated (i.e., posterior) distribution becomes the following mixture:

\begin{equation}
    p(\mathbf{x}_{t_k}\mid \mathbf{z}_{t_{1:k}})
    =
    \sum_{i=1}^{M} w_i^+ \pi_i^+(\mathbf{x}_{t_k})
\label{eq:updated_state_dis}
\end{equation}
where
\begin{equation*}
w_i^+=\frac{\hat{w}_i}{c_k}\left(\int p(\mathbf{z}_{t_k}\mid\mathbf{x}_{t_k}) \hat{\pi}_i(\mathbf{x}_{t_k}) d\mathbf{x}_{t_k} \right),\,\
\pi_i^+(\mathbf{x}_{t_k}) = \frac{p(\mathbf{z}_{t_k}\mid\mathbf{x}_{t_k})\hat{\pi}_i(\mathbf{x}_{t_k})}{\int p(\mathbf{z}_{t_k}\mid\mathbf{x}_{t_k}) \hat{\pi}_i(\mathbf{x}_{t_k}) d\mathbf{x}_{t_k}}.
\end{equation*}

Therefore, the measurement-update step consists of two parts: computing the posterior distribution $\pi_i^+(\mathbf{x}_{t_k})$ for each particle and computing the corresponding weight $w_i^+$. These parts are presented in Sec.~\ref{sec:singlepm} and Sec.~\ref{sec:weightupdate}, respectively.

\subsubsection{Single-particle measurement update} \label{sec:singlepm}
For the single-particle measurement update, we used the measurement-update method from the square-root CD-CKF \cite{arasaratnam2010cubature} because it can accommodate the positive semi-definite matrix calculated during the single-particle propagation step \cite{kim2023wearable, wang2021level}. We restate the method in Algorithm~\ref{alg:spmu} because our notation differs from that in previous studies \cite{arasaratnam2010cubature, kim2023wearable, wang2021level}. Note that this method yields a Gaussian particle, i.e., $\pi_i^+$ in \eqref{eq:updated_state_dis} is Gaussian, because the CD-CKF assumes Gaussianity. This assumption is reasonable in our algorithm, since each particle is expected to have a relatively narrow covariance due to both particle splitting and the measurement update. Consequently, the updated distribution \eqref{eq:updated_state_dis} becomes a Gaussian mixture, providing a suitable setting for performing the next time-update step.

\begin{algorithm}[tb]
\caption{Single-Particle Measurement Update} 
\label{alg:spmu}
\begin{algorithmic}[1]

\REQUIRE A predicted Gaussian particle $\hat{\mathbf{x}}_i \sim \mathcal{N}(\hat{\boldsymbol{\mu}}_i, \hat{\boldsymbol{\Sigma}}_i)$ where $\hat{\boldsymbol{\mu}}_i\in \mathbb{R}^d$ and $\hat{\boldsymbol{\Sigma}}_i = \hat{\mathbf{M}}_i\hat{\mathbf{M}}_i^T$, a measurement $\mathbf{z}_{t_k}$, a measurement function $\mathbf{h}$, and a measurement noise covariance matrix $\mathbf{R}$
\STATE Construct the concatenated matrix of the cubature points:
\[\mathbf{N} = \hat{\boldsymbol{\mu}} + \sqrt{d}\left(\hat{\mathbf{M}_i} \mid -\hat{\mathbf{M}_i}\right).\]
\vspace{-1em}
\STATE Compute the predicted measurement:
\[
\bar{\mathbf{z}}_{t_k} = \frac{1}{2d}\sum_{j=1}^{2d}\mathbf{h}(\mathbf{N})_{j},
\]
where $\mathbf{h}(\mathbf{N})_{j}$ is the $j$-th column of $\mathbf{h}(\mathbf{N})$.
\STATE Perform the QR-factorization:
\[
\begin{pmatrix}
\mathbf{T}_{11} & \mathbf{O} \\
\mathbf{T}_{21} & \mathbf{T}_{22}
\end{pmatrix}
=
\operatorname{qr}
\begin{pmatrix}
\mathbf{h}(\mathbf{N}) & \sqrt{\mathbf{R}} \\
\mathbf{N} & \mathbf{O}
\end{pmatrix},
\]
where $\mathbf{O}$ is a zero matrix.
\STATE Compute the updated mean $\boldsymbol{\mu}_i^+$ and covariance matrix $\mathbf{\Sigma}_i^+$:
\[
\boldsymbol{\mu}_i^+ = \hat{\boldsymbol{\mu}}_i + \mathbf{T}_{21}\mathbf{T}_{11}^{-1}(\mathbf{z}_{t_k} - \bar{\mathbf{z}}_{t_k}), \hspace{0.2cm} \mathbf{\Sigma}_i^+ = \mathbf{T}_{22} \mathbf{T}_{22}^T.
\]
\vspace{-1em}
\RETURN $\mathbf{x}_i^+ \sim \mathcal{N}({\boldsymbol{\mu}}_i^+, {\boldsymbol{\Sigma}}_i^+)$ with PDF $\pi_i^+$
\end{algorithmic}
\end{algorithm}

\subsubsection{Weight update} \label{sec:weightupdate}
Analytically computing the updated weight $w_i^+$ in \eqref{eq:updated_state_dis} is typically intractable \cite{arasaratnam2010cubature}. We thus compute $w_i^+$ by numerically approximating the integral using the cubature rule \cite{stroud1971approximate, cools1997constructing}:
\begin{equation}
    \int p(\mathbf{z}_{t_k}\mid\mathbf{x}_{t_k}) \hat{\pi}_i(\mathbf{x}_{t_k}) d\mathbf{x}_{t_k} 
    \approx \frac{1}{2d} \sum_{j=1}^{2d} \pi_{i,j}(\mathbf{z}_{t_{k}})
\label{eq:cubaapprox}
\end{equation}
where $\pi_{i,j}$ is a Gaussian PDF with mean $h(\hat{\boldsymbol{\mu}}_i + \sqrt{\hat{\mathbf{\Sigma}}_i}\boldsymbol{\xi}_j)$ and covariance matrix $\mathbf{R}$, and the cubature points $\boldsymbol{\xi}_j$ are given by
\begin{equation*} \label{eq:cubature_points_def}
    \boldsymbol{\xi}_j = 
    \begin{cases} 
        \sqrt{d}\,\mathbf{e}_j & j=1, \dots, d \\ 
        -\sqrt{d}\,\mathbf{e}_{j-d} & j=d+1, \dots, 2d.
    \end{cases}
\end{equation*}
Note that $h$ and $\mathbf{R}$ are the measurement function and measurement noise covariance matrix in \eqref{eq:meas}, respectively, and $\mathbf{e}_i$ is the $i$-th unit vector in $\mathbb{R}^d$. Specifically, applying the cubature-rule-based approximation \eqref{eq:cubaapprox} yields 
\begin{align}
    w_i^+ &=\hat{w}_i\frac{\int p(\mathbf{z}_{t_k}\mid\mathbf{x}_{t_k}) \hat{\pi}_i(\mathbf{x}_{t_k}) d\mathbf{x}_{t_k}}{\sum_{l=1}^{M} \hat{w}_l \int p(\mathbf{z}_{t_k} \mid \mathbf{x}_{t_k})\hat{\pi}_l(\mathbf{x}_{t_k}) \, d\mathbf{x}_{t_k}} \nonumber\\
    &\approx \hat{w}_i 
    \frac{\sum_{j=1}^{2d} \pi_{i,j}(\mathbf{z}_{t_k})}
    {\sum_{l=1}^{M} \hat{w}_l \sum_{j=1}^{2d} \pi_{l,j}(\mathbf{z}_{t_k})} \propto \hat{w}_i \sum_{j=1}^{2d} \pi_{i,j}(\mathbf{z}_{t_k})
\label{eq:weight_update}
\end{align}
Based on this proportional relationship in \eqref{eq:weight_update} and the normalization condition $\sum_{i=1}^{M} w_i^+ = 1$, the weights are updated. The detailed procedure is illustrated in Algorithm~\ref{alg:weight_update}.

\begin{algorithm}[H]
\caption{Weight Update}
\label{alg:weight_update}
\begin{algorithmic}[1]

\REQUIRE Predicted weights $\{\hat{w}_i\}_{i=1}^{M}$, a measurement $\mathbf{z}_{t_k}$, a measurement function $h$, a measurement noise covariance matrix $\mathbf{R}$, and predicted means and covariance matrices $\{\hat{\boldsymbol{\mu}}_i,\hat{\mathbf{\Sigma}}_i\}_{i=1}^{M}$ where $\hat{\boldsymbol{\mu}}_i\in \mathbb{R}^d$
\FOR{$i=1$ to $M$}
\STATE Update weights:
    \[
    w_i^+ \leftarrow \hat{w}_i \sum_{j=1}^{2d} \pi_{i,j}(\mathbf{z}_{t_k})
    \]
\ENDFOR
\STATE Normalize weights:
\[
w_i^+ = \frac{w_i^+}{\sum_{l=1}^{M}w^+_l}
\]
\STATE \textbf{return}
$\left\{w_i^+\right\}_{i=1}^{M}$
\end{algorithmic}
\end{algorithm}

\subsection{Implementation details} \label{sec:imple_details}
In this subsection, we explain how Algorithms~\ref{alg:single_particle_update}--\ref{alg:weight_update} are used to implement AMF. In the time-update step, individual particles are propagated using Algorithm~\ref{alg:single_particle_update}. Since the propagations of individual particles are independent, they can be performed in parallel. When particle propagations accumulate errors such that the local linear approximation assumption becomes invalid, the particles are split using Algorithm~\ref{alg:particle_splitting}. Splitting is repeatedly applied until all resulting particles satisfy the error criterion. After propagation and splitting, the increased number of particles is reduced through particle combining described in Algorithm~\ref{alg:kl_gmm_reduction}. Combining is repeated until no more particle pairs can be combined. In the measurement-update step, individual particles and their weights are updated using Algorithms~\ref{alg:spmu} and~\ref{alg:weight_update}, respectively. Then, the updated particles are combined using Algorithm~\ref{alg:kl_gmm_reduction}. The single-particle and weight updates can also be performed in parallel, similar to the single-particle propagations. Including the parallel computations, the overall implementation of AMF is illustrated in Algorithm~\ref{alg:parallel_amf_full}.

\begin{algorithm}[h]
\caption{AMF with Parallel Computation}
\label{alg:parallel_amf_full}
\begin{algorithmic}[1]

\REQUIRE Gaussian particles, time interval $\Delta t$, and measurements $\{\mathbf{z}_{t_k}\}_{k=1}^{K}$
\FOR{$k=1$ to $K$}
    \STATE \textbf{Parallel time-update step:}
    \WHILE{$t_{k-1}\leq$ current time $t<t_{k}$}
    \FORALL{particles \textbf{in parallel}}
        \STATE Propagate particles from $t$ to $\min\{t+\Delta t,t_k\}$ using Algorithm~\ref{alg:single_particle_update}, while performing particle splitting using Algorithm~\ref{alg:particle_splitting}.
    \ENDFOR
    \STATE Combine the propagated particles by partitioning them into grid cells and applying Algorithm~\ref{alg:kl_gmm_reduction} within each cell.
    \STATE $t\leftarrow \min\{t+\Delta t,t_k\}$
    \ENDWHILE
    \STATE \textbf{Parallel measurement-update step:}
    \FORALL{particles \textbf{in parallel}}
        \STATE Update particles using Algorithm~\ref{alg:spmu}
        \STATE Update particle weights using Algorithm~\ref{alg:weight_update}
    \ENDFOR
    \STATE Combine the updated particles by partitioning them into grid cells and applying Algorithm~\ref{alg:kl_gmm_reduction} within each cell.
\ENDFOR
\end{algorithmic}
\end{algorithm}

\subsection{Baseline methods} \label{sec:breview}
In this subsection, we introduce the baselines and explain the rationale for their selection. The first baseline is CD-CKF, which was chosen because it is one of the most widely used and accurate sigma-point filters \cite{arasaratnam2010cubature}. Indeed, prior studies have shown its improved numerical stability and superior performance compared with other sigma-point filters, such as the continuous--discrete unscented Kalman filter \cite{arasaratnam2009cubature, arasaratnam2010cubature}. The second baseline is LSKF, which was selected because it was developed more recently than CD-CKF and has been shown to achieve improved performance in challenging radar-tracking problems \cite{wang2021level}. Furthermore, the LSKF is directly relevant to AMF as it is used for single-particle propagation, illustrated in Algorithm~\ref{alg:single_particle_update}.

Besides these baselines, we included the Gaussian sum accurate continuous–discrete extended cubature Kalman filter (GS-ACD-ECKF) as a third baseline, which estimates the system state using multiple Gaussian particles \cite{wang2018accurate}. This filter was selected because it extends the first baseline, CD-CKF, and has been shown to outperform CD-CKF in radar tracking problems \cite{wang2018accurate}. Because the third baseline employs a fixed number of particles, it is necessary to include an additional baseline that adaptively adjusts the number of particles during filtering. Several adaptive strategies based on particle splitting and combining have been proposed \cite{Vittaldev2016spacecraft, Jones2024physics, Dunik2018Directional, Gutierrez2024Class, Kulik2026nonlinearity, wang2023asymmetric, Demars2013entropy}. Among these methods, we adopt APPD \cite{wang2023asymmetric} because it serves as the primary inspiration for AMF. In APPD, a particle is split using an eigen-decomposition approach, while particle combining is performed based on the proximity of particle means. These operations are applied during the time-update step, while the measurement-update step is performed using Algorithms~\ref{alg:spmu} and~\ref{alg:weight_update}. This APPD-based filter (APPDF) was selected as the fourth baseline because it enables the benefits of our newly proposed particle management strategy to be evaluated in isolation from other algorithmic factors present in other methods \cite{tuggle2018automated, faubel2009split, Faubel2010Further, Volker2011versatile, Horwood2011Adaptive, Yu2014Adaptive}.

\section{Results} \label{sec:Results}

In this section, we demonstrate the performance of AMF through four numerical experiments. The first experiment evaluates the accuracy and computational efficiency of the particle propagation--split--combine scheme (Algorithms~\ref{alg:single_particle_update}--\ref{alg:kl_gmm_reduction}) for uncertainty propagation (i.e., the time update) in the Van der Pol oscillator with a slow–fast manifold. Using the same benchmark, the second experiment evaluates the performance of the full filtering algorithm (Algorithm~\ref{alg:parallel_amf_full}), including the measurement update (Algorithms~\ref{alg:spmu} and~\ref{alg:weight_update}). Third, we test the scalability of AMF in higher-dimensional systems using coupled Van der Pol oscillators. Finally, we evaluate AMF on the Lorenz attractor, a chaotic system with multimodal state PDFs. The hyperparameters, such as $\epsilon_{\text{threshold}}$, $W_{\text{threshold}}$, and $C_{\text{threshold}}$, used in the benchmark studies are listed in Supplementary Table~1.


\subsection{Example 1: Uncertainty propagation in the Van der Pol oscillator} \label{sec:example1}

We demonstrate the performance of the particle propagation--split--combine scheme of AMF by solving the uncertainty propagation problem for the Van der Pol oscillator with asymmetric state PDFs. The system is governed by the following SDE:
\begin{equation} \label{eq:vdp_sde_example1}
    d\begin{bmatrix} y_t \\ z_t \end{bmatrix} = 
    \begin{bmatrix} z_t \\ \mu(1-y_t^2)z_t - y_t \end{bmatrix} dt + 
    \sqrt{\mathbf{K}} d\mathbf{W}_t
\end{equation}
\noindent where $\mu$ is the nonlinearity parameter and $\mathbf{W}_t$ denotes a two-dimensional standard Wiener process. The process noise matrix is set to $\mathbf{K} = 0.006 \mathbf{I}_2$ where $\mathbf{I}_n$ is the $n\times n$ identity matrix. The ground-truth state PDFs were obtained by performing the time-update step of PF via the Euler–Maruyama method with 50,000 particles and a time step of $\Delta t=0.05$, as this configuration was sufficient to achieve convergence (Supplementary Fig.~1). The uncertainty propagation method for the time-update step of AMF is compared with those of LSKF \cite{wang2021level}, CD-CKF \cite{arasaratnam2010cubature}, GS-ACD-ECKF \cite{wang2018accurate}, and APPDF \cite{wang2023asymmetric}. All simulations were conducted until $t=30$ across a range of nonlinearity parameters $\mu \in \{0.6, 0.8, 1.0, 1.2, 1.4, 1.6, 1.8, 2.0\}$. Note that this parameter range covers a significantly nonlinear regime, which is typically not considered in prior studies \cite{KANDEPU2008Applying, Wu2023GNSS, EMZIR2023Multidimensional, EMZIR2024Efficient}.

Figs. \ref{fig:fig5}(a) and \ref{fig:fig5}(b) show that AMF tracks the mean system state more accurately than the baselines. Supplementary Video~1 further illustrates this improvement by visualizing the mean state trajectories over time. This improved performance, quantified by the sum of squared errors (SSE) between the estimated mean trajectories and the PF ground truth, remains consistent across different degrees of nonlinearity (Fig. \ref{fig:fig5}(c)). Importantly, while the SSE of the baselines increases rapidly with increasing nonlinearity ($\mu$), AMF maintains a consistently low and stable error level, highlighting its robustness in capturing nonlinear dynamics.

\begin{figure}[!ht]
    \centering
    \includegraphics[width=\linewidth]{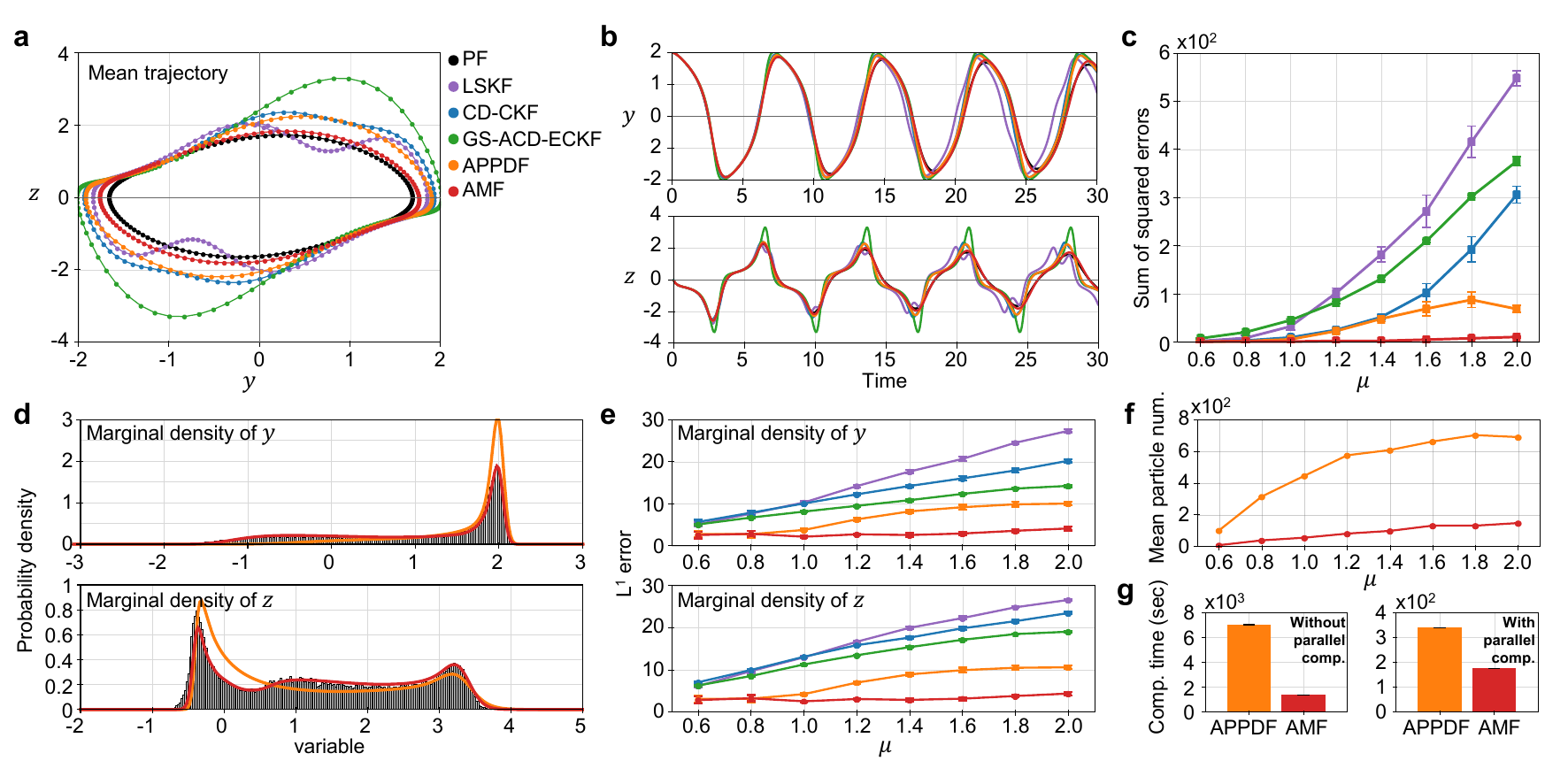}
    \caption{\textbf{Performance of AMF for uncertainty propagation in the Van der Pol oscillator.} 
    \textbf{(a)} Mean state trajectories of AMF, the baselines, and the PF ground truth (GT) for $\mu=1.6$ over $t=20$--$27$.
    \textbf{(b)} Mean trajectories of states $y$ and $z$ over time for $\mu=1.6$.
    \textbf{(c)} SSE between the estimated mean state trajectories and GT for different values of $\mu$.
    \textbf{(d)} Marginal PDFs of states $y$ and $z$ of AMF, APPDF, and GT for $\mu=1.6$ at $t=28.35$.
    \textbf{(e)} $\text{L}^1$ error between the estimated PDFs and GT.
    \textbf{(f)} Mean number of particles during simulation.
    \textbf{(g)} Computational time in sequential and parallel computing settings for $\mu=2.0$ and $\mathbf{K}=0.06\mathbf{I}_2$. Error bars denote the mean $\pm$ standard deviation over five independent runs.}
    \label{fig:fig5}
\end{figure}

Beyond mean propagation, the fidelity of PDF propagation was assessed by comparing the marginal PDFs of states $y$ and $z$. The asymmetric marginal PDFs were captured more accurately by AMF than by APPDF, highlighting the advantages of the particle split--combine scheme proposed in this study (Fig.~\ref{fig:fig5}(d) and Supplementary Video~2). The improvement, quantified by the accumulated $\text{L}^1$ error between the estimated PDFs and the PF ground truth over time, was preserved across different degrees of nonlinearity (Fig. \ref{fig:fig5}(e)). Notably, the accumulated $\text{L}^1$ error for the AMF remained nearly invariant to the increase in $\mu$, whereas the errors for the other methods increased significantly. This confirms that AMF accurately tracks the temporal evolution of non-Gaussian PDFs even under high nonlinearity.

We next compare the mean number of particles during uncertainty propagation between AMF and APPDF. AMF performed uncertainty propagation with fewer particles than APPDF while achieving more accurate propagation (Fig. \ref{fig:fig5}(f) and Supplementary Video~3). Moreover, as the system's nonlinearity increases, the mean number of particles dramatically increases in APPDF. However, AMF shows only a slight increase. This indicates that the number of particles is more effectively regulated by the particle split--combine scheme in AMF (Algorithms~\ref{alg:particle_splitting} and \ref{alg:kl_gmm_reduction}), as guaranteed by \Cref{thm:main_result} and \cite{runnalls2007kullback}. Due to tracking the state PDFs with fewer particles, AMF requires less computational time than APPDF (Fig. \ref{fig:fig5}(g), left). The computational cost can be further dramatically reduced through parallel computing (Fig. \ref{fig:fig5}(g), right), indicating that the algorithmic architecture of AMF is well suited for parallelization. Taken together, these results confirm that AMF offers a practical and high-fidelity solution for state propagation in highly nonlinear systems at a reasonable computational cost.

\subsection{Example 2: Filtering of the Van der Pol oscillator} \label{sec:example2}

We evaluate whether the improved performance of AMF in uncertainty propagation during the time-update step (Fig. \ref{fig:fig5}) is maintained even after measurement updates are performed. The process equation is the same as in Example 1 \eqref{eq:vdp_sde_example1}, and we consider $\mu \in \{1.0, 1.25, 1.5, 1.75, 2.0\}$. Measurements are provided at $t=10$ and $t=20$, and each measurement value is selected to slightly deviate from the ground-truth limit-cycle trajectory in Example 1 to induce a meaningful state correction during the measurement-update step (Supplementary Fig.~2). The relationship between the state $\mathbf{x}_k = [y_k, z_k]^T$ and the measurement vector $\mathbf{z}_k = [r_k, \theta_k]^T$ is defined by the following nonlinear polar coordinate mapping:

\begin{equation} \label{eq:measurement_model1}
    \mathbf{z}_k = h(\mathbf{x}_k) + \mathbf{r}_k, \quad \mathbf{r}_k \sim \mathcal{N}(\mathbf{0}, \mathbf{R})
\end{equation}
\noindent where the measurement function $h(\mathbf{x}_k)$ and the noise covariance matrix $\mathbf{R}$ are given by
\begin{equation} \label{eq:h_and_R}
    h(\mathbf{x}_k) = \begin{bmatrix} \sqrt{y_k^2 + z_k^2} \\ \arctan2(z_k, y_k) \end{bmatrix}, \quad \mathbf{R} = \begin{bmatrix} 0.5^2 & 0 \\ 0 & 0.8^2 \end{bmatrix}.
\end{equation}
\noindent

AMF consistently exhibits improved performance in estimating the filtered mean state across all tested levels of nonlinearity (Figs. \ref{fig:fig6}(a)--\ref{fig:fig6}(c) and Supplementary Video~4). Furthermore, the evolution of state PDFs was filtered more accurately by AMF than by the baselines (Figs. \ref{fig:fig6}(d) and \ref{fig:fig6}(e) and Supplementary Video~5). The improvement becomes more pronounced as the system becomes further nonlinear. In addition to the improvement in accuracy, AMF also shows improved computational efficiency in the full filtering problem. Specifically, AMF tracked the evolution of state PDFs with fewer particles than APPDF (Fig. \ref{fig:fig6}(f) and Supplementary Video~6). AMF thus exhibited reduced computational time (Fig. \ref{fig:fig6}(g), left). The computational time can be further reduced through parallel implementation (Fig. \ref{fig:fig6}(g), right). These results confirm that AMF provides a robust, accurate, and computationally efficient filtering framework for data assimilation in highly nonlinear systems.

\subsection{Example 3: Filtering of the coupled Van der Pol oscillators} \label{sec:example3}

So far, we have confirmed the performance of AMF using a highly nonlinear but $2$--dimensional system (Figs. \ref{fig:fig5} and \ref{fig:fig6}). We next evaluate its performance using higher-dimensional coupled Van der Pol oscillators \cite{Nganso2025White}. The system dynamics are defined by the following $6$--dimensional SDE:

\begin{figure}[!t]
    \centering
    \includegraphics[width=\linewidth]{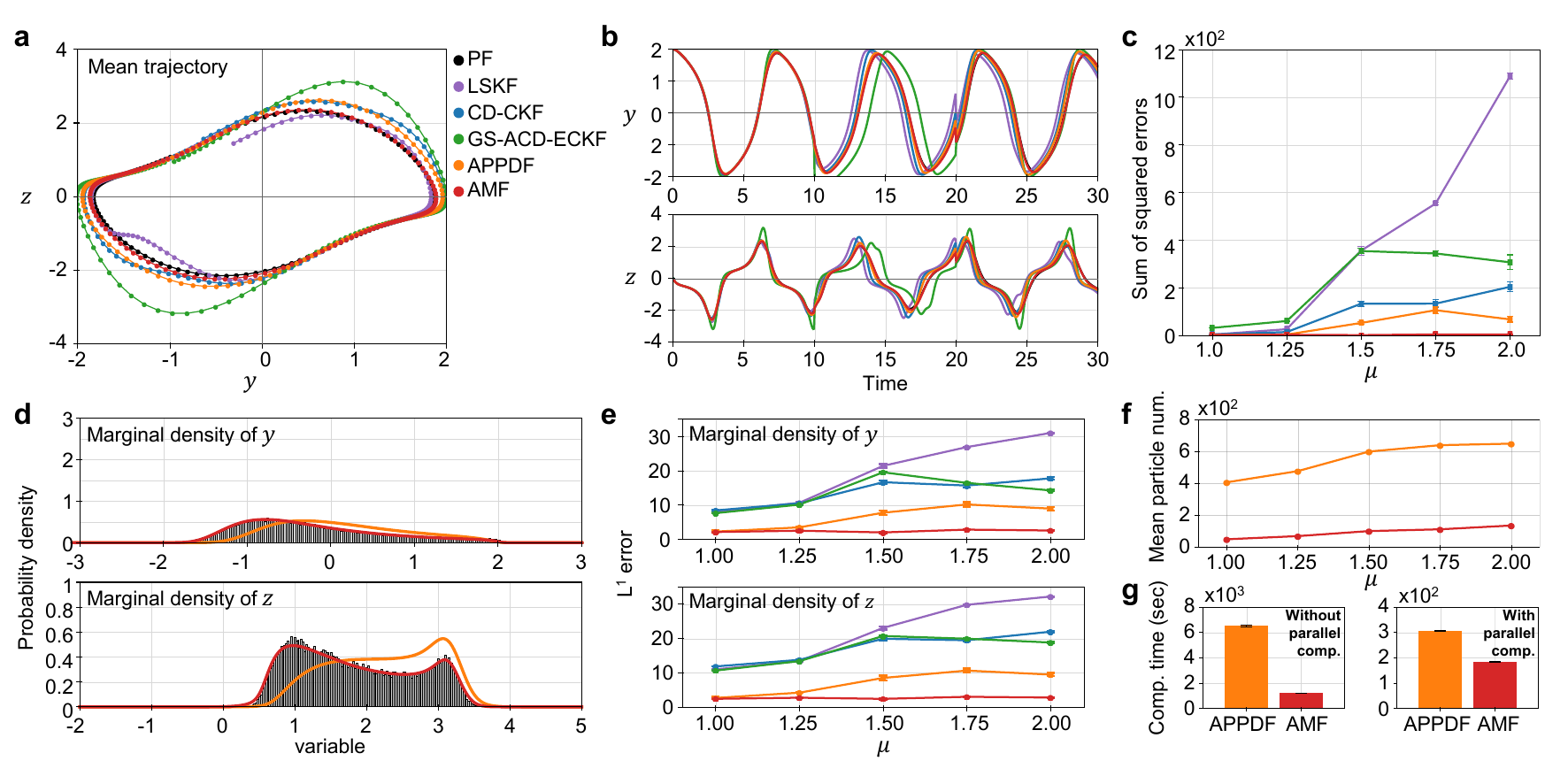}
    \caption{\textbf{Performance of AMF for filtering in the Van der Pol oscillator.} 
    \textbf{(a)} Mean state trajectories of AMF, the baselines, and the PF ground truth (GT) for $\mu=1.5$ over $t=20$--$27$.
    \textbf{(b)} Mean trajectories of states $y$ and $z$ over time for $\mu=1.5$.
    \textbf{(c)} SSE between the estimated mean trajectories and GT for different values of $\mu$.
    \textbf{(d)} Marginal PDFs of states $y$ and $z$ of AMF, APPDF, and GT for $\mu=1.5$ at $t=27.5$.
    \textbf{(e)} $\text{L}^1$ error between the estimated PDFs and GT. AMF maintains a constantly low error across increasing $\mu$.
    \textbf{(f)} Mean number of particles during filtering.
    \textbf{(g)} Computational time in sequential and parallel computing settings for $\mu=2.0$ and $\mathbf{K}=0.06\mathbf{I}_2$. Error bars denote the mean $\pm$ standard deviation over five independent runs.}
    \label{fig:fig6}
\end{figure}

\begin{equation} \label{eq:coupled_vdp_sde}
    d\begin{bmatrix} {y_1}_t \\ {z_1}_t \\ {y_2}_t \\ {z_2}_t \\ {y_3}_t \\ {z_3}_t \end{bmatrix} = 
    \begin{bmatrix} {z_1}_t \\ \mu(1-{y_1}_t^2){z_1}_t - {y_1}_t + k({y_2}_t+{y_3}_t-2{y_1}_t) \\ {z_2}_t \\ \mu(1-{y_2}_t^2){z_2}_t - {y_2}_t + k({y_1}_t+{y_3}_t-2{y_2}_t) \\ {z_3}_t \\ \mu(1-{y_3}_t^2){z_3}_t - {y_3}_t + k({y_1}_t+{y_2}_t-2{y_3}_t)\end{bmatrix} dt + 
    \sqrt{\mathbf{K}} d\mathbf{W}_t
\end{equation}

\noindent where $\mu$ is the nonlinearity parameter, $k$ is the coupling parameter, and $\mathbf{W}_t$ denotes a $6$--dimensional standard Wiener process. The process noise matrix is set to $\mathbf{K} = 0.006 \mathbf{I}_6$. The ground-truth state PDFs were obtained by performing PF with 300,000 particles and a time step of $\Delta t = 0.01$, as this setting was sufficient to achieve convergence (Supplementary Fig.~3). We consider $\mu \in \{1.5, 2.0, 2.5\}$ and $k \in \{0.1, 0.2, 0.3\}$. The filtering ends at $t=10$, and the measurement is provided at $t=5$.
Similar to Example 2 \eqref{eq:measurement_model1}, the relationship between the state $\mathbf{x}_k = [y_{1,k}, z_{1,k}, y_{2,k}, z_{2,k}, y_{3,k}, z_{3,k}]^T$ and the measurement vector $\mathbf{z}_k =[r_{1,k}, \theta_{1,k}, r_{2,k}, \theta_{2,k}, r_{3,k}, \theta_{3,k}]^T$ is defined by the following nonlinear polar coordinate mapping:

\begin{equation} \label{eq:measurement_model2}
    \mathbf{z}_k = h(\mathbf{x}_k) + \mathbf{r}_k, \quad \mathbf{r}_k \sim \mathcal{N}(\mathbf{0}, \mathbf{R})
\end{equation}
\noindent where the measurement function $h(\mathbf{x}_k)$ and the noise covariance matrix $\mathbf{R}$ are given by
\begin{equation} \label{eq:h_and_R2}
    h(\mathbf{x}_k) = \begin{bmatrix} 
    \sqrt{y_{1,k}^2 + z_{1,k}^2} \\ \arctan2(z_{1,k}, y_{1,k}) 
    \\ \sqrt{y_{2,k}^2 + z_{2,k}^2} \\ \arctan2(z_{2,k}, y_{2,k}) 
    \\ \sqrt{y_{3,k}^2 + z_{3,k}^2} \\ \arctan2(z_{3,k}, y_{3,k})
    \end{bmatrix}, \quad \mathbf{R} = 0.2^2 \mathbf{I}_6.
\end{equation}

Fig. \ref{fig:fig7}(a) and Supplementary Video~7 show that AMF tracks the state of the three coupled oscillators more accurately than LSKF and APPDF, on which AMF was developed, for a representative case ($\mu=2$ and $k=0.3$). This improved performance is preserved across different values of the nonlinearity and coupling parameters (Fig. \ref{fig:fig7}(b)). In particular, the SSE of AMF becomes dramatically smaller than those of the other methods as the nonlinearity increases. Notably, the SSE of APPDF grows rapidly and exceeds that of LSKF as the nonlinearity increases (Fig. \ref{fig:fig7}(b)), although the LSKF estimate appears farther from the ground-truth trajectory (Fig. \ref{fig:fig7}(a)). This is due to a phase difference between the APPDF estimate and the ground-truth trajectory, such that similar state values are attained at different times (Supplementary Video~7).

\begin{figure}[!htbp]
    \centering
    \includegraphics[width=\linewidth]{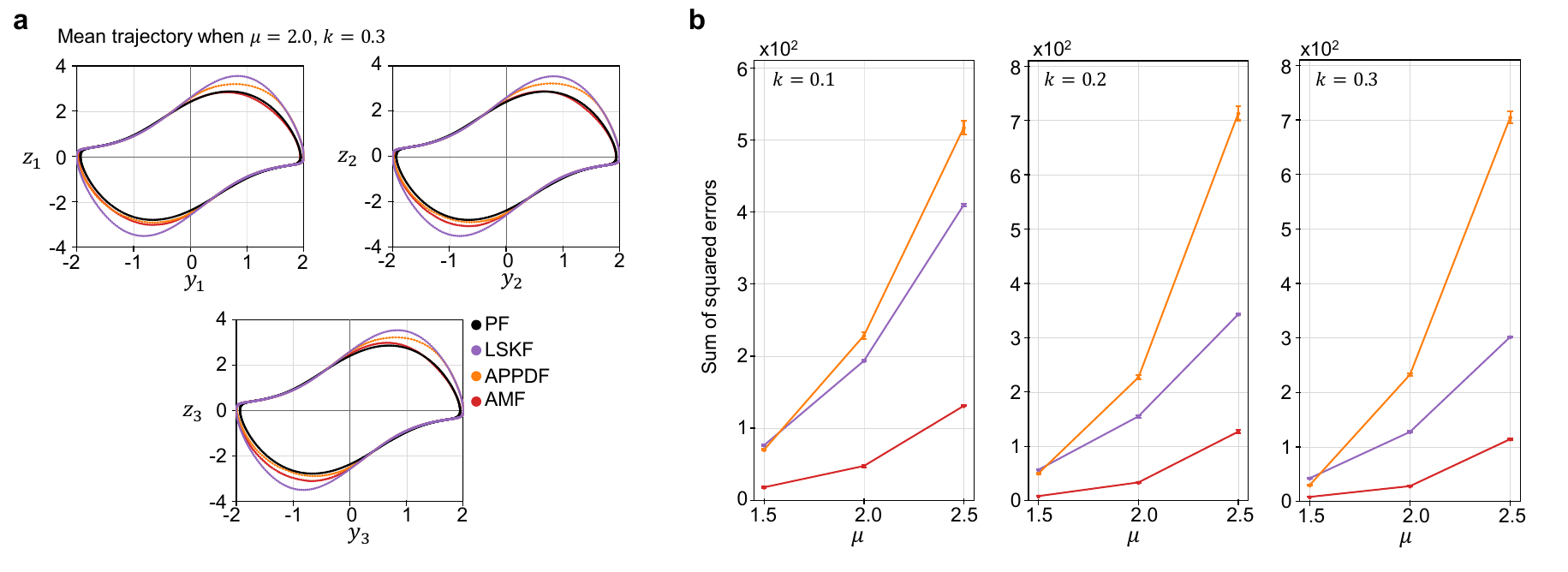}
    \caption{\textbf{Accuracy of AMF for filtering in the coupled Van der Pol oscillators}
    \textbf{(a)} Mean state trajectories of AMF, the baselines, and the PF ground truth (GT) for $\mu=2.0$ and $k=0.3$ over $t=0.5$--$10.0$. 
    \textbf{(b)} SSE between the estimated mean trajectories and GT across different combinations of $\mu$ and $k$. Error bars denote the mean ± standard deviation over five independent runs.}
    \label{fig:fig7}
\end{figure}

The lower accuracy of APPDF compared with LSKF is counterintuitive because APPDF extends LSKF with a particle split--combine scheme based on eigen-decomposition \cite{wang2023asymmetric}. We hypothesize that APPDF leads to excessive particle splitting, resulting in excessive pruning, and ineffective particle combining, leading to the accumulation of numerical errors, whereas AMF does not. To test this hypothesis, we examined changes in the number of particles due to particle splitting ($\Delta N_{\text{split}}$), pruning ($\Delta N_{\text{prune}}$; i.e., removing particles with small weights), and combining ($\Delta N_{\text{combine}}$) during uncertainty propagation in the coupled Van der Pol oscillators (Fig. \ref{fig:fig8}(a)).

Fig. \ref{fig:fig8}(b) shows that the cumulative number of particle splits, $\sum \Delta N_{\mathrm{split}}$, is smaller in AMF than in APPDF. The excessive particle splitting in APPDF generates many small-weight particles, resulting in more particles being pruned than in AMF. AMF thus preserves more information during the pruning step. Furthermore, AMF exhibits a larger $\sum \Delta N_{\mathrm{combine}}$ than APPDF, confirming that it more effectively reduces the number of particles while preserving as much information as possible. These results are also observed in the time series of $\Delta N_{\mathrm{split}}$, $\Delta N_{\mathrm{prune}}$, and $\Delta N_{\mathrm{combine}}$ (Fig. \ref{fig:fig8}(c)). Therefore, AMF requires a lower mean number of particles and lower computational cost than APPDF (Figs. \ref{fig:fig8}(d) and \ref{fig:fig8}(e)), confirming the hypothesis. Specifically, as $\mu$ increases, APPDF performs increasingly excessive particle splitting (Figs. \ref{fig:fig8}(b) and \ref{fig:fig8}(c)). This results in a large number of small-weight particles being generated and subsequently pruned, leading to a reduction in the total number of particles in APPDF (Fig. \ref{fig:fig8}(d)). This reduction results in information loss, which significantly degrades estimation accuracy. In addition, repeated particle splitting accumulates numerical errors. As a result, APPDF becomes less accurate than LSKF in highly nonlinear regimes, while AMF avoids these issues (Fig. \ref{fig:fig7}(b)).

\begin{figure}[!t]
    \centering
    \includegraphics[width=\linewidth]{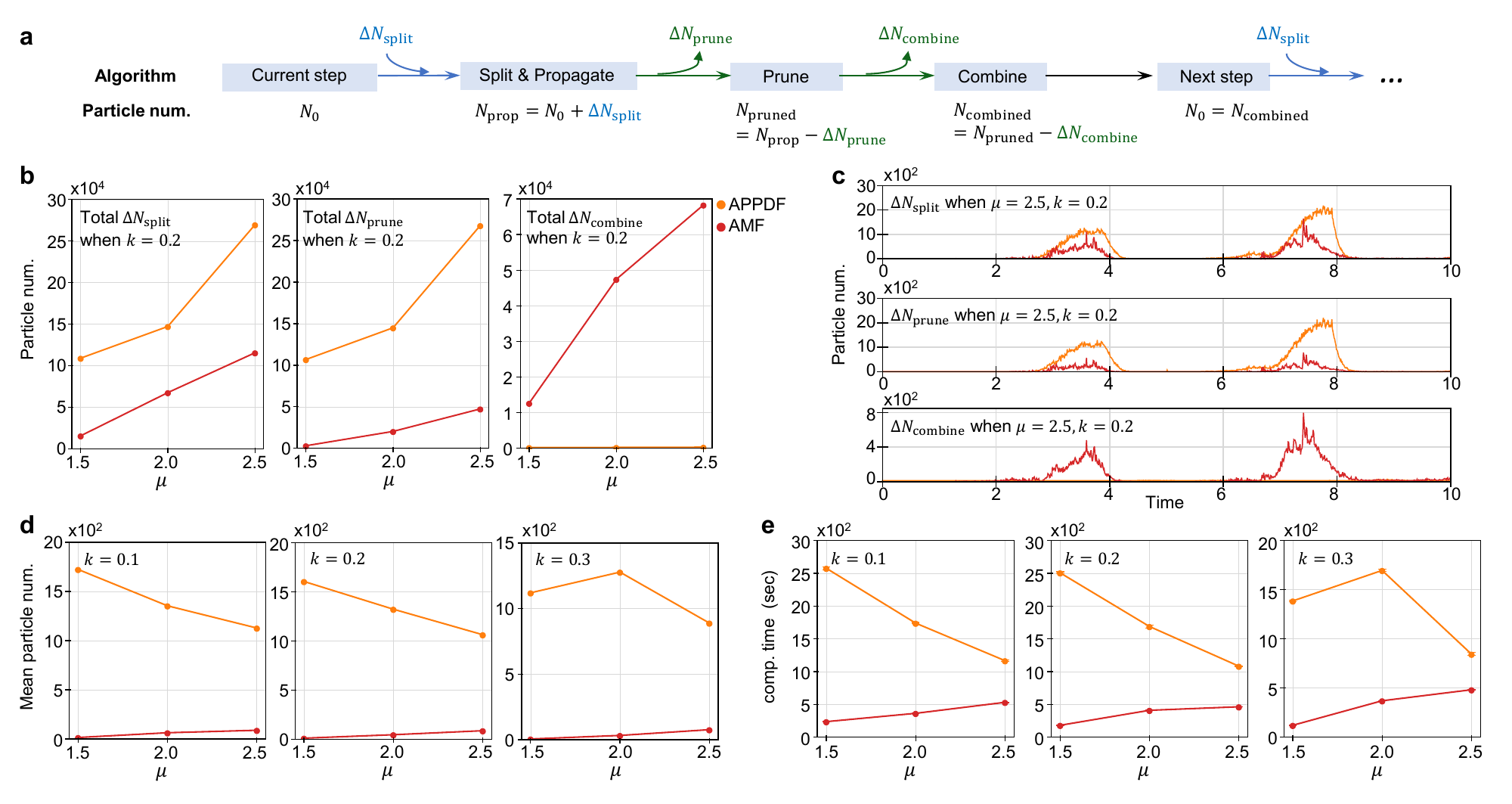}
    \caption{\textbf{Changes in the number of particles due to particle splitting, pruning, and combining.}
    \textbf{(a)} Schematic diagram illustrating the evolution of the particle number during uncertainty propagation using AMF and APPDF. Changes in the particle number due to splitting, pruning, and combining are denoted by $\Delta N_{\mathrm{split}}$, $\Delta N_{\mathrm{prune}}$, and $\Delta N_{\mathrm{combine}}$, respectively.
    \textbf{(b)} Cumulative particle-number changes, $\sum \Delta N_{\mathrm{split}}$, $\sum \Delta N_{\mathrm{prune}}$, and $\sum \Delta N_{\mathrm{combine}}$, for APPDF and AMF across different settings.
    \textbf{(c)} Time series of $\Delta N_{\mathrm{split}}$, $\Delta N_{\mathrm{prune}}$, and $\Delta N_{\mathrm{combine}}$.
    \textbf{(d)} Mean number of particles during simulation.
    \textbf{(e)} Computational time of simulations. Error bars denote the mean ± standard deviation over five independent runs.}
    \label{fig:fig8}
\end{figure}

\subsection{Example 4: Filtering of the Lorenz attractor} 
\label{sec:example4}

We further evaluate AMF using the Lorenz attractor, a representative chaotic system. The system dynamics are described by the following SDE:
\begin{equation} \label{eq:lorenz_sde}
    d
    \begin{bmatrix}
        x_t \\ y_t \\ z_t
    \end{bmatrix}
    =
    \begin{bmatrix}
        \sigma(y_t-x_t) \\
        x_t(\rho-z_t)-y_t \\
        x_ty_t-\beta z_t
    \end{bmatrix} dt
    +
    \sqrt{\mathbf{K}}d\mathbf{W}_t
\end{equation}
where $\sigma=3$, $\rho=26.5$, $\beta=1$, and $\mathbf{W}_t$ denotes a three-dimensional standard Wiener process. The process noise matrix is set to $\mathbf{K}=0.4\mathbf{I}_3$. The ground-truth state PDFs were obtained by performing PF with 500,000 particles and a time step of $\Delta t = 0.01$, as this setting was sufficient to achieve convergence (Supplementary Fig.~4). The simulation ends at $t=20$, and a single measurement is provided at $t=10$. The measurement equation that maps the state vector $\mathbf{x}_k$ to the measurement vector $\mathbf{z}_k$ is given by
\begin{equation}
\mathbf{z}_k = \mathbf{x}_k + \mathbf{r}_k,
\quad
\mathbf{r}_k \sim \mathcal{N}(\mathbf{0}, 0.8^2\mathbf{I}_3).
\end{equation}

Figs. \ref{fig:fig9}(a)--\ref{fig:fig9}(d) show heat maps of the state PDFs estimated by PF, LSKF, APPDF, and AMF, respectively, projected onto the $xy$, $yz$, and $zx$ planes at a representative time point $t=12$. The corresponding marginal PDFs along the $x$, $y$, and $z$ dimensions are compared in Fig. \ref{fig:fig9}(e). The results demonstrate that AMF most accurately captures the asymmetric and multimodal state PDFs. The improved accuracy of AMF is further confirmed by comparing the time evolution of the estimated state PDFs (Supplementary Video~8). For quantitative evaluation, the SSE between the mean trajectory estimated by each method and the ground-truth mean trajectory was computed (Fig. \ref{fig:fig9}(f)). In addition, the $\text{L}^1$ error of the marginal PDFs was calculated for each method (Fig. \ref{fig:fig9}(g)). These results further confirm that AMF provides the most accurate estimates of both the mean trajectory and the state distribution.

Next, we compare the mean number of particles and the total computational time between APPDF and AMF (Fig. \ref{fig:fig9}(h)). This confirms that the proposed split--combine scheme of AMF dramatically reduces computational cost while improving filtering accuracy, consistent with the results obtained for the Van der Pol oscillator (Figs. \ref{fig:fig5}--\ref{fig:fig8}). Taken together, these results demonstrate the potential of AMF to accurately track the state of chaotic systems with asymmetric and multimodal PDFs.

\begin{figure}[!t]
    \centering
    \includegraphics[width=\linewidth]{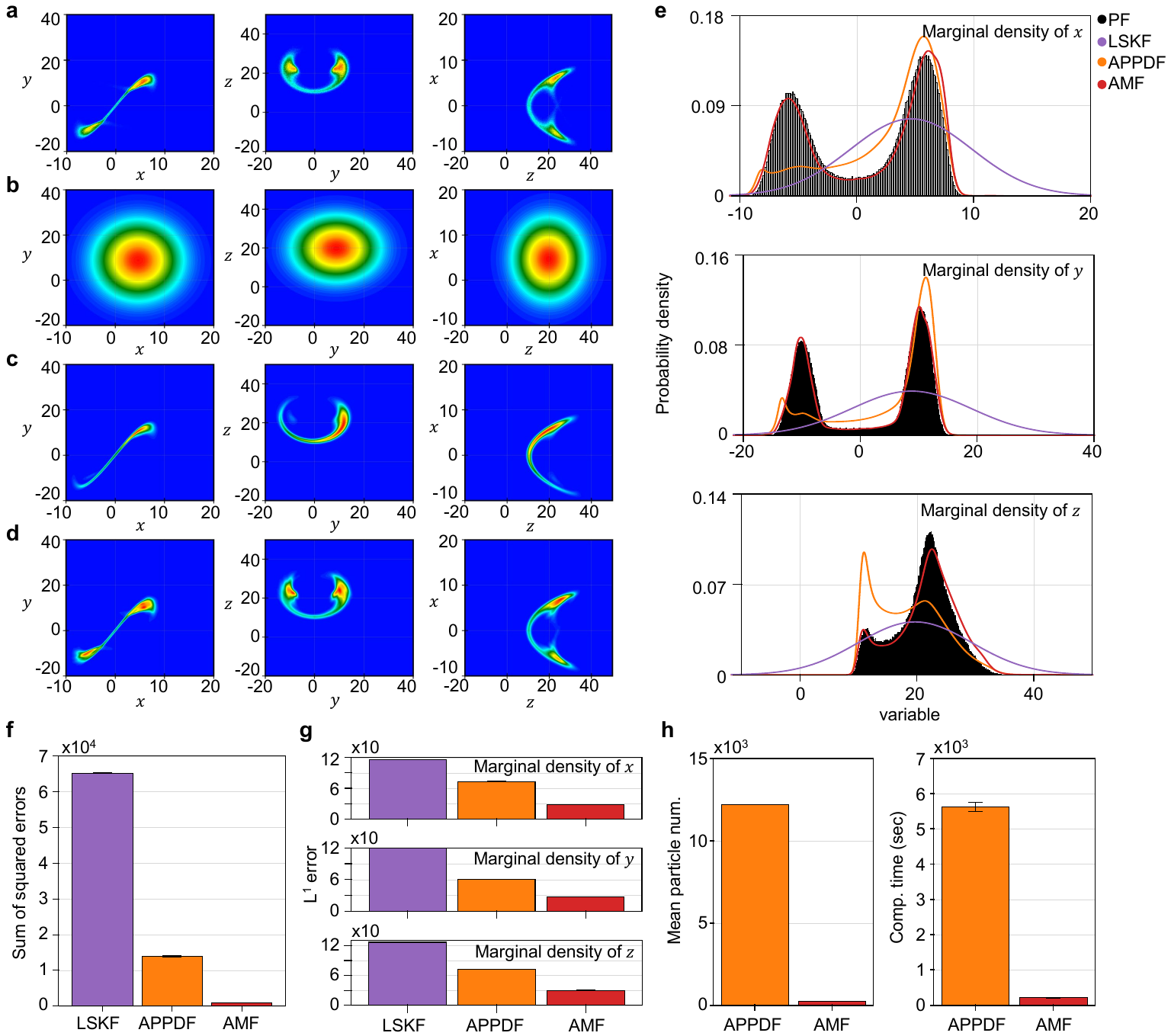}
    \caption{\textbf{Performance of AMF for filtering in the Lorenz attractor.}
    \textbf{(a)--(d)} Heat maps of the state PDFs projected onto the $xy$-, $yz$-, and $zx$-planes for PF (a), LSKF (b), APPDF (c), and AMF (d), respectively, at $t=12$.
    \textbf{(e)} Marginal PDFs of states $x$, $y$, and $z$ at $t=12$.
    \textbf{(f)} SSE between the estimated mean trajectories and the PF ground truth (GT).
    \textbf{(g)} $\text{L}^1$ error between the estimated marginal PDFs and GT.
    \textbf{(h)} Mean number of particles during filtering (left) and computational time (right) for APPDF and AMF. Error bars denote the mean ± standard deviation over five independent runs.}
    \label{fig:fig9}
\end{figure}

\section{Conclusion} \label{sec:Conclusion}
In this work, we developed AMF for estimating asymmetric and multimodal posterior densities of the system state over time (Fig. \ref{fig:fig1}). The core of AMF is an adaptive split–combine scheme that dynamically adjusts the number and weights of Gaussian particles during both time-update and measurement-update steps. In particular, the splitting method guarantees a reduction in variance along a target level-set-point direction of a particle (\Cref{thm:main_result}) without any online auxiliary optimization (Algorithm~\ref{alg:particle_splitting}). Moreover, the combining method efficiently evaluates the similarity between densities using the upper bound on the KL divergence (Algorithm~\ref{alg:kl_gmm_reduction}). Taken together, these features lead to improved performance in terms of both accuracy and computational efficiency (Figs. \ref{fig:fig5}--\ref{fig:fig9}).

AMF overcomes several limitations of existing filters and uncertainty propagation methods. Specifically, prior methods perform particle splitting along the eigenvector direction, which only indirectly reduces local nonlinearity in the target level-set-point direction \cite{faubel2009split, wang2023asymmetric}. This can result in excessive particle splitting, the accumulation of splitting errors, and consequently reduced filtering accuracy (Figs. \ref{fig:fig5}--\ref{fig:fig9}). In addition, particle combining based solely on mean proximity can introduce numerical errors (Fig. \ref{fig:fig3}). These limitations are addressed by AMF (\Cref{thm:main_result} and Algorithms~\ref{alg:particle_splitting}--\ref{alg:weight_update}). Furthermore, AMF improves computational efficiency (Figs. \ref{fig:fig5}(g) and \ref{fig:fig6}(g)) through a parallel computing framework (Algorithm~\ref{alg:parallel_amf_full}), exploiting the independence of split-and-propagate operations across particles.

AMF outperformed existing methods across challenging benchmarks with asymmetric and multimodal state PDFs (Figs. \ref{fig:fig5}--\ref{fig:fig9}). Notably, it achieved improved accuracy with fewer particles, highlighting the effectiveness of the proposed split–combine scheme. We also showed the scalability of AMF to high-dimensional systems by successfully filtering $6$--dimensional coupled oscillators using CPU-based parallelization (Fig. \ref{fig:fig7}). This scalability could be further extended to even higher-dimensional complex systems, such as neuronal population dynamics, through a GPU-based implementation \cite{rodriguez2015opencl}, representing a promising direction for future research.

Despite the progress made in this study, there remains room for improvement in filtering methods, including AMF. First, the process equation is assumed to have a constant noise level over time. Relaxing this assumption is particularly important for biological systems; for example, the level of biochemical reaction noise depends on both the molecular copy numbers \cite{schnoerr2017approximation}. One potential approach is to employ the Lamperti transform, which reformulates the SDE into an equivalent form with constant diffusion \cite{moller2010state}. Second, filtering methods generally require a known governing SDE; however, in many practical applications, the underlying dynamics are not available $a$ $priori$. To address this, data-driven inference of the latent SDE, using methods such as SDE Matching \cite{bartosh2025sde}, should precede filtering. Third, measurement updates in ultra-high-dimensional systems (e.g., with more than $10^5$ state variables) could become inaccurate because the Gaussian-assumed cubature rule may lose accuracy in such settings. Thus, developing scalable measurement-update methods remains a challenge. One promising avenue is the integration of score-based generative modeling into filtering, as it enables efficient sampling from high-dimensional PDFs \cite{rozet2023score, bao2024score, yoon2026rethinking}.

We showed that AMF accurately estimates the time evolution of state PDFs in both the Van der Pol oscillator and the Lorenz system. Indeed, variants of these systems have been widely used to model real-world nonlinear phenomena across multiple scales, ranging from neuronal dynamics to atmospheric convection. This suggests that AMF has the potential to address real-world data assimilation problems. For example, it could be applied to estimate endogenous circadian rhythms \cite{kim2023wearable, mayer2025predicting, huang2021predicting} and use these estimates to assess human health conditions \cite{lee2024real, lim2024accurately}. Furthermore, AMF could be used for weather forecasting \cite{van2009particle} and to enhance the spatial and temporal resolution of neural data \cite{woolrich2009bayesian}.

Beyond data assimilation, AMF also provides an accurate framework for simulating uncertainty propagation (Figs. \ref{fig:fig5} and \ref{fig:fig8}). This capability could be leveraged for continuous-time generative modeling, where sample generation is formulated as the evolution of probability densities over time \cite{song2020score, lipman2022flow, albergo2025stochastic}. Taken together, these findings demonstrate that AMF provides a flexible and scalable framework for modeling probability density evolution in nonlinear dynamical systems.

\section{CRediT authorship contribution statement} \label{sec:act}
\textbf{San Kim}: Methodology, Software, Validation, Formal analysis, Investigation, Writing--original draft, Writing--review \& editing, Visualization; \textbf{Won Chang}: Conceptualization, Writing--review \& editing, Funding acquisition; \textbf{Daniel B. Forger}: Conceptualization, Writing--review \& editing, Funding acquisition; \textbf{Dae Wook Kim}: Conceptualization, Methodology, Validation, Formal analysis, Investigation, Writing--original draft, Writing--review \& editing, Visualization, Supervision, Project administration, Funding acquisition.

\section{Data availability} \label{sec:dai}
All code used to perform the simulations will be made freely available on GitHub upon acceptance of the manuscript.

\section{Declaration of competing interest} \label{sec:dci}
The authors declare that they have no known competing financial interests or personal relationships that could have appeared to influence the work reported in this paper.

\section{Acknowledgments} \label{sec:ack}
This work was supported by the National Research Foundation of Korea (Grant No. RS-2023-00301976, RS-2025-00523567, RS-2025-00561696, RS-2025-02215354, RS-2026-25512726, and RS-2026-25522728), the Korea Basic Science Institute (Grant No. RS-2026-25500300), the New Faculty Startup Fund from KAIST (Grant No. G04240060), the New Faculty Startup Fund from Seoul National University (Grant No. 326-20240027), and the U.S. National Science Foundation (Grant No. DMS-2052499).


\bibliographystyle{elsarticle-num} 
\bibliography{references}
\clearpage

\setcounter{figure}{0}
\renewcommand{\figurename}{Supplementary Figure}
\renewcommand{\thefigure}{\arabic{figure}}

\begin{table}[H]
\centering
\small
\setlength{\tabcolsep}{3pt}
\begin{tabular}{lccccc}
\toprule
System
& \makecell{$\epsilon_{\mathrm{threshold}}$}
& \makecell{$W_{\mathrm{threshold}}$}
& \makecell{$C_{\mathrm{threshold}}$}
& \makecell{$r$}
& \makecell{Grid cell \\ volume} \\
\midrule
\makecell[l]{Van der Pol oscillator}
& $0.05$ & $10^{-4}$ & $0.03$ & $0.005$ & $0.04^2$ \\
\makecell[l]{Coupled Van der Pol\\ oscillator}
& $0.05$ & $10^{-4}$ & $0.05$ & $0.005$ & $1^6$ \\
Lorenz attractor 
& $0.05$ & $10^{-5}$ & $0.05$ & $0.05$ & $1^3$ \\
\bottomrule
\end{tabular}
\caption*{{Supplementary Table 1:} \textbf{Hyperparameter settings used in the numerical experiments.} $\epsilon_{\mathrm{threshold}}$ and $W_{\mathrm{threshold}}$ are thresholds for particle splitting and pruning, respectively, used in both AMF and APPDF. $C_{\mathrm{threshold}}$ is the combining threshold used in AMF, and $r$ is the combining radius used in APPDF. The grid cell volume used to discretize the effective support of the state space for particle combining in both AMF and APPDF is reported. The value of $\epsilon_{\mathrm{threshold}}$ is adopted from Wang and Forger (Ref.~42 in the main text), which inspired this study. The value of $W_{\mathrm{threshold}}$ is chosen such that a particle with weight 1 can be split at least 6 and at most 20 times. In the Lorenz attractor case, $W_{\mathrm{threshold}}$ is set to a smaller value because the effective support of the state space is larger than that of the Van der Pol oscillator (Supplementary Videos~2, 5, and 8). $C_{\mathrm{threshold}}$, $r$, and the grid cell volume are selected according to the size of the effective support in the state space to ensure appropriate particle combining for each system.}
\end{table}
\clearpage

\begin{figure}[H]
    \centering
    \includegraphics[width=\linewidth]{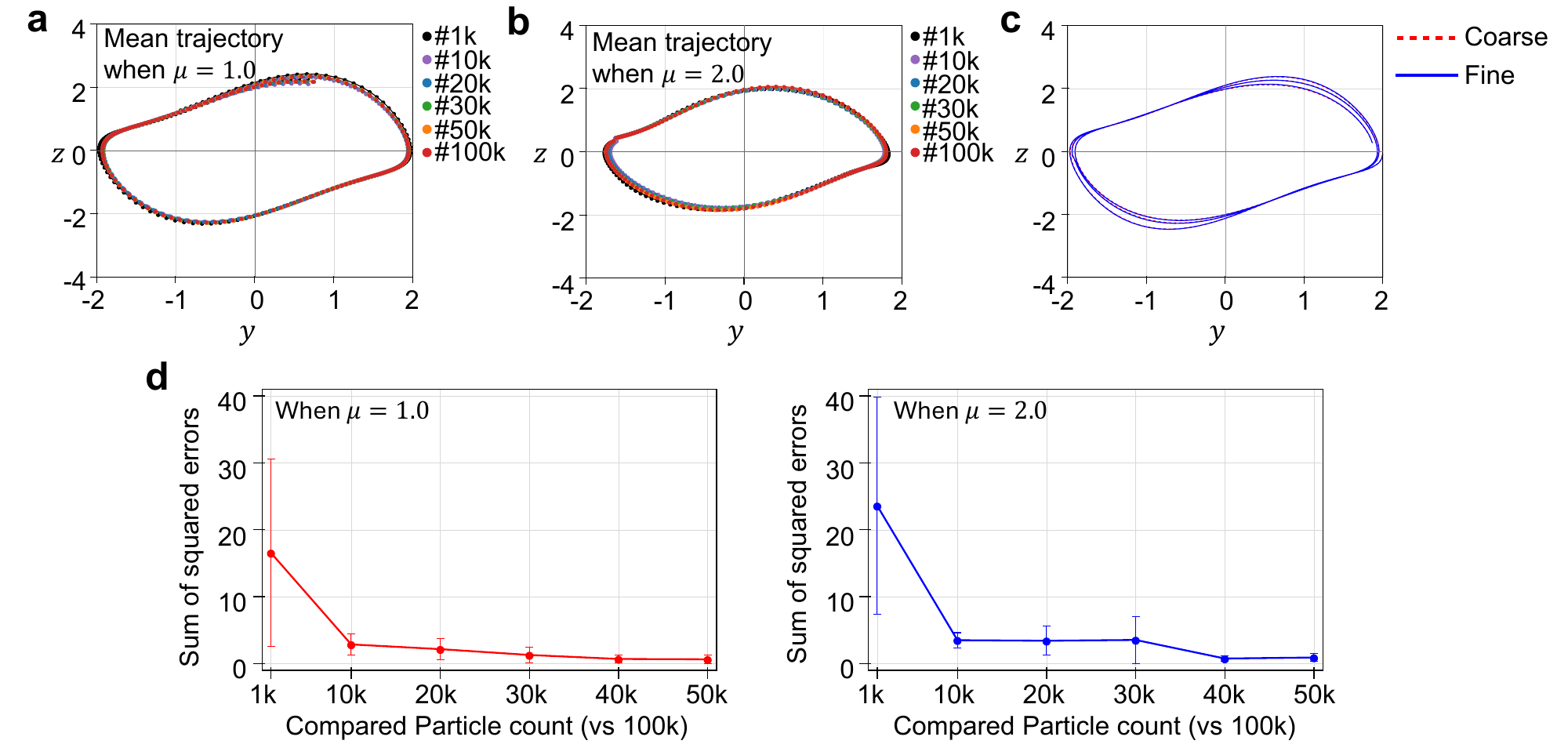}
    \caption{\textbf{Convergence of PF for the Van der Pol oscillator.} To assess the convergence of PF, we examine how the mean state trajectory depends on the number of particles and the numerical integration time step for the Van der Pol oscillator with $\mu=1.0$ and $\mu=2.0$. \textbf{(a,\hspace{0.02cm}b)} Mean state trajectories obtained with 1k, 10k, 20k, 30k, 50k, and 100k particles using $dt=0.05$, plotted over $t=5$--$12.5$, for $\mu=1.0$ (a) and $\mu=2.0$ (b). \textbf{(c)} Mean state trajectories computed using coarse ($dt=0.05$) and fine ($dt=0.001$) integration time steps with 50k particles, plotted over $t=0$--$20$. The coarse and fine trajectories nearly overlap, with a sum of squared error (SSE) of 0.924 between the two trajectories. \textbf{(d)} SSE between trajectories obtained with 1k, 10k, 20k, 30k, and 50k particles and the 100k-particle reference trajectory over $t=0$--$20$ for $\mu=1.0$ (left) and $\mu=2.0$ (right). Error bars denote the mean $\pm$ standard deviation over ten independent runs. For both values of $\mu$, the SSE decreases as the number of particles increases and becomes nearly zero from 40k particles onward. Based on these convergence results, we use 50k particles and $dt=0.05$ to implement PF in Examples 1 and 2.}
    \label{fig:supple_fig1}
\end{figure}

\begin{figure}[H]
    \centering
    \includegraphics[width=\linewidth]{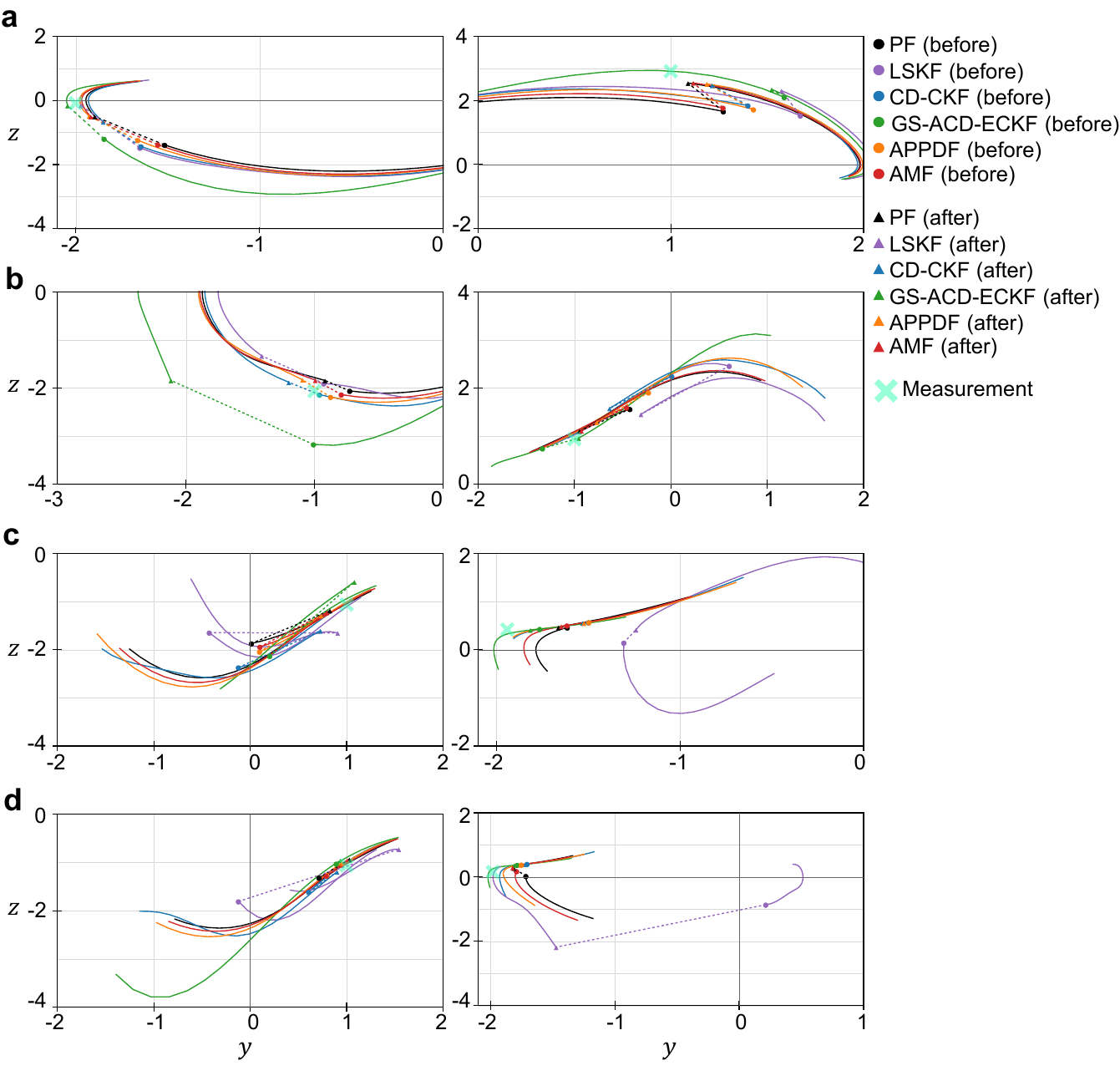}
    \caption{\textbf{Effect of the measurement update on filter trajectories in the Van der Pol oscillator.} Filter trajectories are shown for different values of the nonlinearity parameter $\mu$: (a) $\mu=1.25$, (b) $\mu=1.5$, (c) $\mu=1.75$, and (d) $\mu=2.0$. For each $\mu$, measurements are assimilated at $t=10$ and $t=20$. The subplots in the first and second columns show trajectories over $t=9$--$11$ and $t=19$--$21$, respectively. In each subplot, the measurement is marked by a single cross. Along each filter trajectory, the mean states immediately before and after the measurement update are marked by a circle and a triangle, respectively. Dotted lines connect the pre-measurement-update and post-measurement-update states, visualizing the displacement induced by the measurement update. This shows that the measurement update leads to a non-negligible state correction. This indicates that the measurement-update step also contributes to the filtering performance, in addition to the time-update step, in our experimental settings.}
    \label{fig:supple_fig2}
\end{figure}
\clearpage

\begin{figure}[H]
    \centering
    \includegraphics[width=\linewidth]{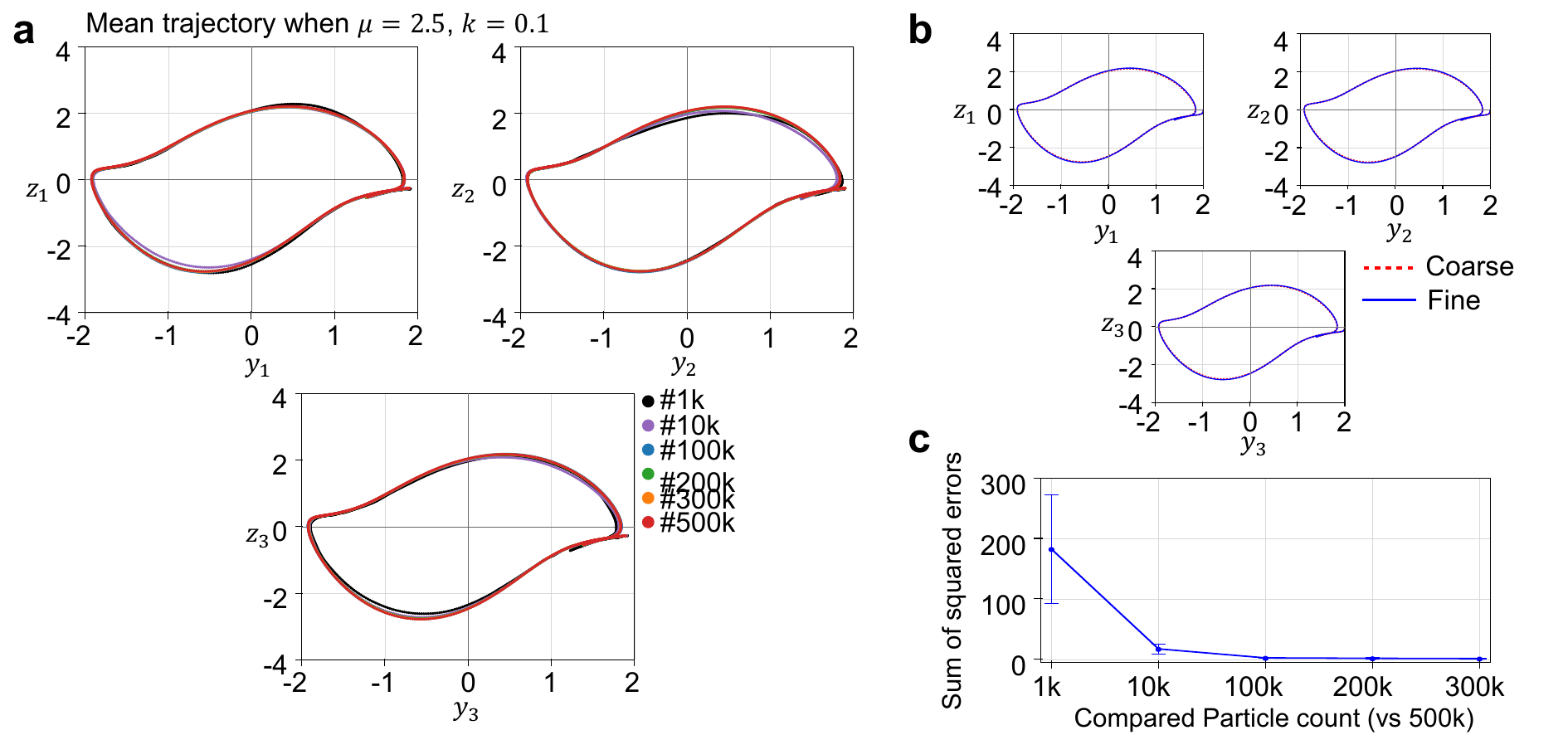}
    \caption{\textbf{Convergence of PF for the coupled Van der Pol oscillators.} To assess the convergence of PF, we examine how the mean state trajectory depends on the number of particles and the numerical integration time step for the coupled Van der Pol oscillators with $\mu=2.5$ and $k=0.1$, corresponding to the most nonlinear case considered in Example 3. \textbf{(a)} Mean state trajectories of each oscillator obtained with 1k, 10k, 100k, 200k, 300k, and 500k particles using $dt=0.01$, plotted over $t=0.5$--$10$. \textbf{(b)} Mean state trajectories computed using coarse ($dt=0.01$) and fine ($dt=0.0005$) integration time steps with 300k particles, plotted over $t=0$--$10$. The coarse and fine trajectories nearly overlap, with a SSE of 1.796. \textbf{(c)} SSE between trajectories obtained with 1k, 10k, 100k, 200k, and 300k particles and the 500k-particle reference trajectory over $t=0$--$10$ for $\mu=2.5$ and $k=0.1$. The SSE decreases as the number of particles increases and becomes nearly zero from 100k particles onward. Based on these convergence results, we use 300k particles and $dt=0.01$ to implement PF in Example 3.}
    \label{fig:supple_fig3}
\end{figure}
\clearpage

\begin{figure}[H]
    \centering
    \includegraphics[width=\linewidth]{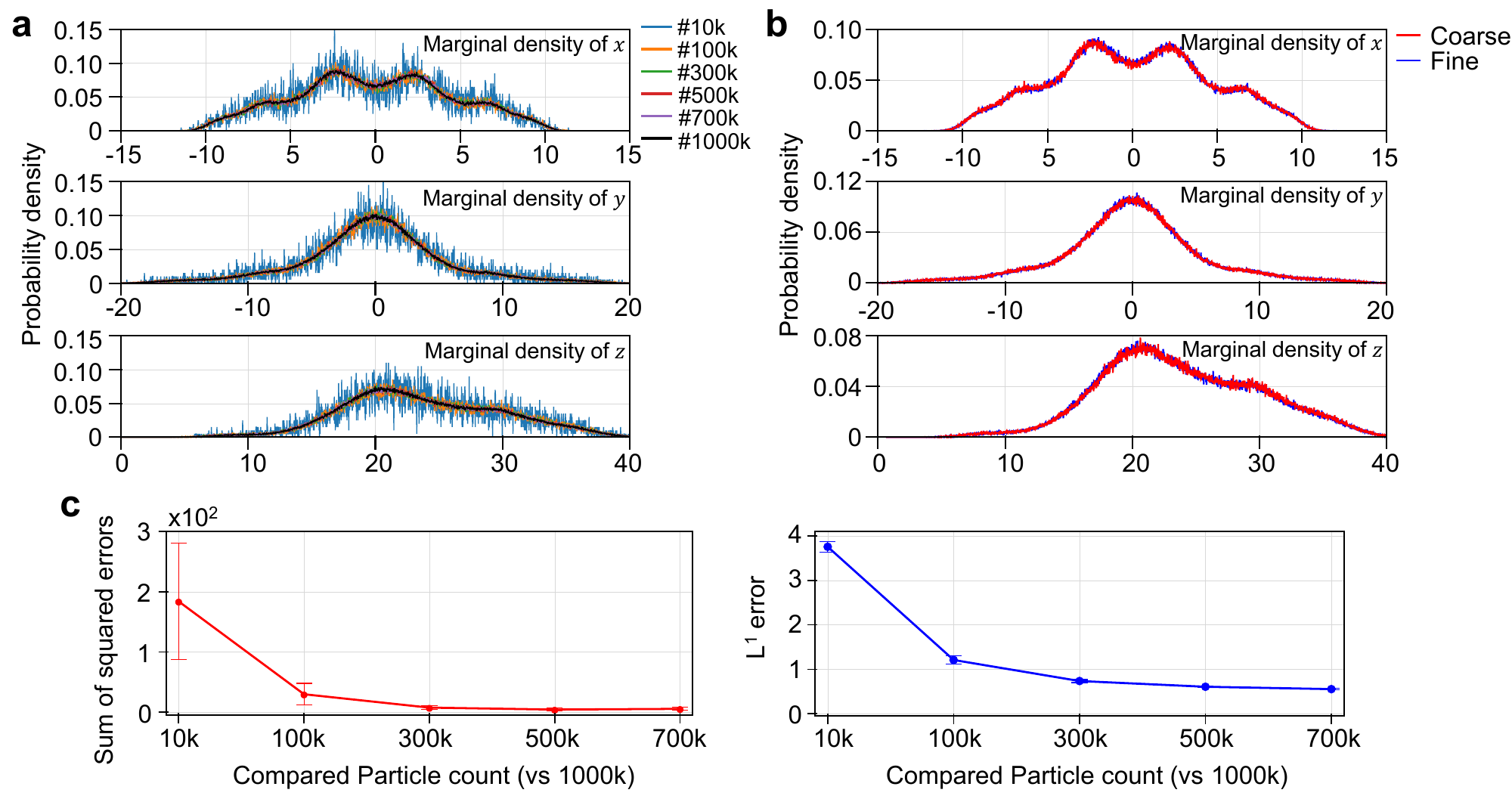}
    \caption{\textbf{Convergence of PF for the Lorenz Attractor.} To assess the convergence of PF, we examine how the mean state trajectory and marginal densities of $x$, $y$, and $z$ depend on the number of particles and the numerical integration time step for the Lorenz attractor. \textbf{(a)} Marginal densities of $x$, $y$, and $z$ obtained with 10k, 100k, 300k, 500k, 700k, and 1000k particles using $dt=0.01$, plotted at $t=15$. \textbf{(b)} Marginal densities of $x$, $y$, and $z$ computed using coarse ($dt=0.01$) and fine ($dt=0.0005$) integration time steps with 500k particles, plotted at $t=15$. The SSE between the coarse and fine mean state trajectories is 10.78, and the sum of $\text{L}^1$ norm between the marginal densities at $t=10,11,\ldots,20$ is 0.715, indicating close agreement between the coarse and fine simulations. \textbf{(c)} SSE between mean state trajectories (left) and $\text{L}^1$ errors between the marginal densities of $x$, $y$, and $z$ (right), computed using 10k, 100k, 300k, 500k, and 700k particles relative to the 1000k-particle reference result over $t=0$--$20$. The $\text{L}^1$ error is computed by summing the $\text{L}^1$ differences between the marginal densities at $t=10,11,\ldots,20$, as in (b). Both the SSE and the $\text{L}^1$ error converge from 300k particles onward. Based on these convergence results, we use 500k particles and $dt=0.01$ to implement PF in Example 4.}
    \label{fig:supple_fig4}
\end{figure}
\clearpage

\captionsetup{type=figure}

\captionof*{figure}{{Supplementary Video 1.} \textbf{Temporal evolution of mean state trajectories for uncertainty propagation in the Van der Pol oscillator.} Mean state trajectories of AMF, the baselines, and the PF ground truth for $\mu=1.6$ over $t=0$--$30$.}

\captionof*{figure}{{Supplementary Video 2.} \textbf{Temporal evolution of joint and marginal PDFs for uncertainty propagation in the Van der Pol oscillator.} Joint PDFs of the PF ground truth, APPDF, and AMF are shown in the top-left, top-middle, and top-right panels, respectively. Comparisons of the $y$- and $z$-marginal PDF are shown in the bottom-left and bottom-right panels, respectively. Results are shown for $\mu=1.6$ over $t=0$--$30$.}

\captionof*{figure}{{Supplementary Video 3.} \textbf{Temporal evolution of the level sets and the number of particles for uncertainty propagation in the Van der Pol oscillator.} Adaptive changes in the level sets of APPDF and AMF are shown in the top-left and bottom-left panels, respectively. Adaptive changes in the particle numbers of APPDF and AMF are shown in the top-right and bottom-right panels, respectively. Results are shown for $\mu=1.6$ over $t=0$--$30$.}

\captionof*{figure}{{Supplementary Video 4.} \textbf{Temporal evolution of mean state trajectories for filtering in the Van der Pol oscillator.} Mean state trajectories of AMF, the baselines, and the PF ground truth for $\mu=1.5$ over $t=0$--$30$.}

\captionof*{figure}{{Supplementary Video 5.} \textbf{Temporal evolution of joint and marginal PDFs for filtering in the Van der Pol oscillator.} Joint PDFs of the PF ground truth, APPDF, and AMF are shown in the top-left, top-middle, and top-right panels, respectively. Comparisons of the $y$- and $z$-marginal PDF are shown in the bottom-left and bottom-right panels, respectively. Results are shown for $\mu=1.5$ over $t=0$--$30$.}

\captionof*{figure}{{Supplementary Video 6.} \textbf{Temporal evolution of the level sets and the number of particles for filtering in the Van der Pol oscillator.} Adaptive changes in the level sets of APPDF and AMF are shown in the top-left and bottom-left panels, respectively. Adaptive changes in the particle numbers of APPDF and AMF are shown in the top-right and bottom-right panels, respectively. Results are shown for $\mu=1.5$ over $t=0$--$30$.}

\captionof*{figure}{{Supplementary Video 7.} \textbf{Temporal evolution of mean state trajectories for filtering in the coupled Van der Pol oscillator.} Mean state trajectories of AMF, the baselines, and the PF ground truth for $\mu=2.0$ and $k=0.3$ over $t=0$--$10$.}

\captionof*{figure}{{Supplementary Video 8.} \textbf{Temporal evolution of joint PDFs for filtering in the Lorenz attractor.} The $xy$-, $yz$-, and $zx$-joint PDFs are shown for the PF ground truth, LSKF, APPDF, and AMF from top to bottom, respectively, over $t=0$--$20$.}


\appendix
\numberwithin{equation}{section}

\section{Gaussian level-set invariance}
\label{app:app1}

Let us assume that the initial probability density function $u(\mathbf{x},0)$ is Gaussian. The factorization $\mathbf{M}(0)$ of the initial covariance matrix $\boldsymbol{\Sigma}(0)$ is denoted by:
\begin{equation} \label{eq:factorization}
    \mathbf{M}(0) = \begin{bmatrix} \mathbf{x}_1(0) & \cdots & \mathbf{x}_d(0) \end{bmatrix}.
\end{equation}
It is straightforward to show that the column vectors $\mathbf{x}_i(0)$ lie on the same level set:
\begin{align*}
    \mathbf{x}_i(0)^T \boldsymbol{\Sigma}(0)^{-1} \mathbf{x}_i(0) &= \mathbf{x}_i(0)^T \mathbf{M}(0)^{-T} \mathbf{M}(0)^{-1} \mathbf{x}_i(0) \\
    &= \left( \mathbf{M}(0)^{-1} \mathbf{x}_i(0) \right)^T \left( \mathbf{M}(0)^{-1} \mathbf{x}_i(0) \right) \\
    &= \mathbf{e}_i^T \mathbf{e}_i = 1,
\end{align*}
\noindent where $\mathbf{e}_i$ is the $i$-th standard basis vector. We next consider the evolution of these vectors. Under the local linear approximation $\mathbf{v(x)} \approx \mathbf{Jx}$, the velocity field $\mathbf{v}_L$ is linear with respect to $\mathbf{x}$, as described in Sec. \ref{sec:spp} of the main text. Hence, the flow map from time $0$ to $t$ can be represented by a linear transition matrix $\mathbf{A}(t)$, such that $\mathbf{x}_i(t) = \mathbf{A}(t)\mathbf{x}_i(0)$. The covariance matrix $\boldsymbol{\Sigma}(t)$ then evolves via congruence transformation:
\begin{equation} \label{eq:cov_update}
    \boldsymbol{\Sigma}(t) = \mathbf{M}(t)\mathbf{M}(t)^T = \left(\mathbf{A}(t)\mathbf{M}(0)\right) \left(\mathbf{A}(t)\mathbf{M}(0)\right)^T = \mathbf{A}(t) \boldsymbol{\Sigma}(0) \mathbf{A}(t)^T.
\end{equation}
By substituting $\mathbf{x}_i(t)$ and $\boldsymbol{\Sigma}(t)$ into the quadratic form, we observe that the value is invariant:
\begin{align*}
    \mathbf{x}_i(t)^T \boldsymbol{\Sigma}(t)^{-1} \mathbf{x}_i(t) 
    &= \left( \mathbf{A}(t)\mathbf{x}_i(0) \right)^T \left( \mathbf{A}(t) \boldsymbol{\Sigma}(0) \mathbf{A}(t)^T \right)^{-1} \left( \mathbf{A}(t)\mathbf{x}_i(0) \right) \\
    &= \mathbf{x}_i(0)^T \mathbf{A}(t)^T \mathbf{A}(t)^{-T} \boldsymbol{\Sigma}(0)^{-1} \mathbf{A}(t)^{-1} \mathbf{A}(t) \mathbf{x}_i(0) \\
    &= \mathbf{x}_i(0)^T \boldsymbol{\Sigma}(0)^{-1} \mathbf{x}_i(0)\\
    &= 1.
\end{align*}
Therefore, tracking the propagation of a Gaussian particle is equivalent to tracking its level set.

\section{Interpretation of the numerator in \eqref{eq:relative_error}}
\label{app:app2}
Here, we show that the numerator in \eqref{eq:relative_error} represents the sum of all higher-order ($\geq2$) terms arising from the Taylor expansion of the drift function $\mathbf{v}$, scaled by a constant factor. Let $\mathbf{v}:\mathbb{R}^d \to \mathbb{R}^d$ be a drift function and let $\Delta\mathbf{x}^{\hat{\imath}}$ be denoted by $\boldsymbol{\delta}$ for brevity. The Taylor expansion of $\mathbf{v}(\mathbf{x}+\boldsymbol{\delta})$ around $\mathbf{x}$ can be written as
\begin{equation} \label{eq:taylor_drift}
\mathbf{v}(\mathbf{x}+\boldsymbol{\delta})
=
\mathbf{v}(\mathbf{x})
+
\sum_{n=1}^{\infty}
\frac{1}{n!}
D^n\mathbf{v}(\mathbf{x})
[\boldsymbol{\delta},\ldots,\boldsymbol{\delta}]
\end{equation}
where $D^n\mathbf{v}(\mathbf{x})$ denotes the $n$-th order derivative, which is an $n$-linear map
\[
D^n\mathbf{v}(\mathbf{x}) :
(\mathbb{R}^d)^n \to \mathbb{R}^d,
\]
defined by
\[
D^n\mathbf{v}(\mathbf{x})[\mathbf{h}_1,\ldots,\mathbf{h}_n]
=
\sum_{i_1,\ldots,i_n=1}^d
\frac{\partial^n \mathbf{v}}{\partial x_{i_1}\cdots\partial x_{i_n}}(\mathbf{x})
\,h_{1,i_1}\cdots h_{n,i_n}.
\]

We now perform a Taylor expansion of each term appearing in the numerator of \eqref{eq:relative_error}.

\textbf{1. The first term:} Substituting $2\boldsymbol{\delta}$ into \eqref{eq:taylor_drift} yields
\begin{align}
\frac12\bigl(\mathbf{v}(\mathbf{x}+2\boldsymbol{\delta})-\mathbf{v}(\mathbf{x})\bigr)
&=
\frac12
\sum_{n=1}^{\infty}
\frac{1}{n!}
D^n\mathbf{v}(\mathbf{x})
[2\boldsymbol{\delta},\ldots,2\boldsymbol{\delta}] \\
&=
\sum_{n=1}^{\infty}
\frac{2^{\,n-1}}{n!}
D^n\mathbf{v}(\mathbf{x})
[\boldsymbol{\delta},\ldots,\boldsymbol{\delta}].
\end{align}

\textbf{2. The second term:} \eqref{eq:taylor_drift} yields
\begin{equation}
\mathbf{v}(\mathbf{x}+\boldsymbol{\delta})-\mathbf{v}(\mathbf{x})
=
\sum_{n=1}^{\infty}
\frac{1}{n!}
D^n\mathbf{v}(\mathbf{x})
[\boldsymbol{\delta},\ldots,\boldsymbol{\delta}].
\end{equation}

Hence, the numerator in \eqref{eq:relative_error} is given by:
\begin{align}
\left\| \left[
\frac12(\mathbf{v}(\mathbf{x}+2\boldsymbol{\delta})-\mathbf{v}(\mathbf{x}))
\right]
-
\left[
\mathbf{v}(\mathbf{x}+\boldsymbol{\delta})-\mathbf{v}(\mathbf{x})
\right] \right\|\\
=
\left\| \sum_{n=1}^{\infty}
\left(
\frac{2^{n-1}-1}{n!}
\right)
D^n\mathbf{v}(\mathbf{x})
[\boldsymbol{\delta},\ldots,\boldsymbol{\delta}] \right\|.
\end{align}

Because $\left(\frac{2^{1-1}-1}{1!}\right)=0$, the linear Jacobian term vanishes. Therefore, the numerator contains all nonlinear terms except for the linear term.

\section{Invariance of the particle-splitting error}
\label{app:app3}
Here, we show that the error introduced by particle splitting \eqref{eq:final_particles} is identical to the error obtained from the corresponding $1$-dimensional particle splitting optimization problem \eqref{eq:split_opt}. Let us consider the random vector $\vec{y} = \mathbf{T}\vec{x}= \begin{pmatrix} y_1 \\ \vec{y}_{\perp} \end{pmatrix} \sim \mathcal{N}(\vec{0}, \mat{\Sigma}_\vec{y})$ \eqref{eq:y_partition}, where $y_1 \in \mathbb{R}$, $\vec{y}_{\perp} \in \mathbb{R}^{n-1}$, and the covariance matrix $\mat{\Sigma}_\vec{y} = \begin{pmatrix} \sigma_d^2 & \vec{0}^T \\ \vec{0} & \mat{\Sigma}_{\vec{y}, {\perp}} \end{pmatrix}$ \eqref{eq:sigma_prime_block}. 
The marginal distributions are Gaussian, with $y_1 \sim \mathcal{N}(0, \sigma_d^2)$ and $\vec{y}_{\perp} \sim \mathcal{N}(\vec{0}, \mat{\Sigma}_{\vec{y}, {\perp}})$. Since $y_1$ and $\vec{y}_{\perp}$ are jointly Gaussian and uncorrelated, they are independent. This allows for the factorization of the joint probability density function (PDF):
\begin{equation} \label{eq:joint_1}
    K(\vec{y}; \vec{0}, \mat{\Sigma}_{\vec{y}}) = K(y_1; 0, \sigma_d^2)K(\vec{y}_{\perp}; \vec{0}, \mat{\Sigma}_{\vec{y},\perp})
\end{equation}
Similarly, the PDF of a single split particle in the transformed space is: 
\begin{equation} \label{eq:joint_2}
    K(\vec{y}; \vec{0}, \mat{\Sigma}_{\vec{y},\text{split}}) = K(y_1; 0, \sigma_d^2/2)K(\vec{y}_{\perp}; \vec{0}, \mat{\Sigma}_{\vec{y},\perp})    
\end{equation}
The difference between the PDF of the parent particle $\vec{y}$ and the sum of the PDFs of the three child particles, denoted by $f_Y(\vec{y})$, is:
\begin{multline}
    f_Y(\vec{y}) := K(\vec{y}; \vec{0}, \mat{\Sigma}_{\vec{y}}) \\
    - \left[ (1-2w)K(\vec{y}; \vec{0}, \mat{\Sigma}_{\vec{y},\text{split}}) + w(K(\vec{y}; \vec{a}, \mat{\Sigma}_{\vec{y},\text{split}}) + K(\vec{y}; -\vec{a}, \mat{\Sigma}_{\vec{y},\text{split}})) \right]
\end{multline}
where $\vec{a}$ is $[a, 0, \dots, 0]^T$, $a=1.03332\sigma_d$, and $w=0.21921$. Then, the $\text{L}^1$ error, denoted by $\varepsilon_\vec{y}$, is given by:
\begin{equation}
    \varepsilon_\vec{y} := \int_{\mathbb{R}^{n}} |f_Y(\vec{y})|d\vec{y}
\end{equation}
Similarly, the $\text{L}^1$ error of the corresponding $1$-dimensional optimization problem, denoted by $\varepsilon_{\mathrm{single}}$, is given by:
\begin{multline} 
    \varepsilon_{\text{single}}:=\int_{\mathbb{R}} \left| K(x \mid 0, \sigma_d^2) \right. \\
    \left. - \left( (1-2w)K(x \mid 0, \frac{\sigma_d^2}{2}) + wK(x \mid -a, \frac{\sigma_d^2}{2}) + wK(x \mid a, \frac{\sigma_d^2}{2}) \right) \right|dx
\end{multline}
Note that we use the $\text{L}^1$ error here instead of the maximum difference error in \eqref{eq:split_opt} because it allows analytical error analysis, and the difference between the errors calculated using the parameter values obtained by solving \eqref{eq:split_opt} is negligible (approximately $10^{-5}$).

First, we show that $\varepsilon_\vec{y}=\varepsilon_{\text{single}}$, which implies that the high-dimensional reconstruction error is entirely determined by the corresponding $1$-dimensional particle splitting optimization problem. By \eqref{eq:joint_1}, \eqref{eq:joint_2}, and Fubini's theorem, we have:
\begin{equation} \label{eq:decoup_single}
\begin{split}
    \varepsilon_\vec{y} &= \int_{\mathbb{R}^{n}} |f_Y(\vec{y})|d\vec{y} \\
    &= \int_{\mathbb{R}^{n-1}}\int_{\mathbb{R}} \left| K(y_1; 0, \sigma_d^2) - \left[ (1-2w)K\left(y_1; 0, \frac{\sigma_d^2}{2}\right) + w(\dots) \right] \right| \\
    &\quad \times \left| K(\vec{y}_{\perp}; \vec{0}, \mat{\Sigma}_{\perp}) \right| dy_1 d\vec{y}_{\perp} \\
    &= \int_{\mathbb{R}^{n-1}}\left| K(\vec{y}_{\perp}; \vec{0}, \mat{\Sigma}_{\perp}) \right| d\vec{y}_{\perp} \\
    &\quad \times \int_{\mathbb{R}} \left| K(y_1; 0, \sigma_d^2) - \left[ (1-2w)K\left(y_1; 0, \frac{\sigma_d^2}{2}\right) + w(\dots) \right] \right| dy_1 \\
    &= 1 \times \varepsilon_{\text{single}} = \varepsilon_{\text{single}} 
\end{split}
\end{equation}

Next, we show that $\varepsilon_{\text{single}}$ is equal to the $\text{L}^1$ error computed for the random variable $\vec{x}$ in the original space, denoted by $\varepsilon_\vec{x}$. Because $\vec{y} = \mat{T}\vec{x}$, we have $d\vec{x} = |\det(\mat{T})|^{-1} d\vec{y}$, and the difference between the PDF of the parent particle and the sum of the PDFs of the three child particles satisfies $f_X(\vec{x}) = f_Y(\vec{y}) |\det(\mat{T})|$. Therefore,
\begin{align}
    \varepsilon_\vec{x}
    &:=\int_{\mathbb{R}^n} \left| f_X(\vec{x}) \right|d\vec{x} \nonumber\\
    &=\int_{\mathbb{R}^n} \left( \left| f_Y(\vec{y}) \right| \left|\det\mat{T}\right| \right) \left( \frac{1}{\left| \det\mat{T}\right|}d\vec{y} \right) \nonumber \\
    &=\int_{\mathbb{R}^n} \left| f_Y(\vec{y})\right|d\vec{y} \nonumber\\
    &=\varepsilon_\vec{y}=\varepsilon_{\text{single}} 
\end{align}
This provides a rationale for using the parameters $w$ and $a$, derived from \eqref{eq:split_opt}, to split a multivariate Gaussian distribution of $\vec{x}$.

\section{Derivation of the KL divergence upper bound}
\label{app:app4}

Here, we derive \eqref{eq:kl_bound1} and \eqref{eq:kl_bound_final}, presented in Sec. \ref{sec:pc} of the main text. Let
\[
    w_{\text{com}}=w_1+w_2,
    \qquad
    \pi_{\text{com}}(x)=w_{\text{com}}\mathcal{N}(x;\boldsymbol{\mu}_{\text{com}},\mathbf{\Sigma}_{\text{com}})
\]
denote the weighted density of the combined particle, where
\[
    \boldsymbol{\mu}_{\text{com}}=\frac{w_1\boldsymbol{\mu}_1+w_2\boldsymbol{\mu}_2}{w_{\text{com}}},
\]
and
\begin{equation}
\label{eq:app_combined_cov}
\begin{aligned}
    \mathbf{\Sigma}_{\text{com}}
    &=
    \frac{w_1}{w_{\text{com}}}
    \left[
    \mathbf{\Sigma}_1
    +
    (\boldsymbol{\mu}_1-\boldsymbol{\mu}_{\text{com}})
    (\boldsymbol{\mu}_1-\boldsymbol{\mu}_{\text{com}})^T
    \right]
    \\
    &\quad+
    \frac{w_2}{w_{\text{com}}}
    \left[
    \mathbf{\Sigma}_2
    +
    (\boldsymbol{\mu}_2-\boldsymbol{\mu}_{\text{com}})
    (\boldsymbol{\mu}_2-\boldsymbol{\mu}_{\text{com}})^T
    \right].
\end{aligned}
\end{equation}
Since direct computation of KL divergence \eqref{eq:kl_div} is intractable, a computable upper bound is employed instead. 
The upper bound is derived using the three lemmas.

\begin{lemma} \label{thm:lemma_1}
If $p_1(x)$, $p_2(x)$ and $q(x)$ are any PDFs over $d$-dimensions, $0 \le w \le 1$ and $\bar{w}=1-w$, then:
\begin{equation} \label{eq:theorem1}
    D_{\text{KL}}(wp_1 + \bar{w}q, wp_2 + \bar{w}q) \le w D_{\text{KL}}(p_1, p_2)
\end{equation}
\end{lemma}

\begin{lemma} \label{thm:lemma_2}
If $p(x)$, $q_1(x)$ and $q_2(x)$ are any PDFs over $d$-dimensions, $0 \le w \le 1$ and $\bar{w}=1-w$, then:
\begin{align} 
    D_{\text{KL}}(wq_1 + \bar{w}q_2, p) &\le w D_{\text{KL}}(q_1, p) + \bar{w} D_{\text{KL}}(q_2, p) \label{eq:theorem2a} \\
    D_{\text{KL}}(p, wq_1 + \bar{w}q_2) &\le w D_{\text{KL}}(p, q_1) + \bar{w} D_{\text{KL}}(p, q_2) \label{eq:theorem2b}
\end{align}
\end{lemma}

\begin{lemma} \label{thm:lemma_3}
Let $g_1(x)$ and $g_2(x)$ be the $d$-dimensional Gaussian PDFs with mean vectors $\boldsymbol{\mu}_1, \boldsymbol{\mu}_2$ and positive-definite covariance matrices $\mathbf{\Sigma}_1, \mathbf{\Sigma}_2$, respectively. Then:
\begin{equation} \label{eq:gaussian_kl}
    2D_{\text{KL}}(g_1, g_2) = \mathrm{tr}(\mathbf{\Sigma}_2^{-1}[\mathbf{\Sigma}_1 - \mathbf{\Sigma}_2 + (\boldsymbol{\mu}_1 - \boldsymbol{\mu}_2)(\boldsymbol{\mu}_1 - \boldsymbol{\mu}_2)^T]) + \log\left(\frac{\det(\mathbf{\Sigma}_2)}{\det(\mathbf{\Sigma}_1)}\right)
\end{equation}
\end{lemma}

To apply Lemma~\ref{thm:lemma_1}, we let $w=w_{\text{com}}=w_1+w_2$ and $p_1=\frac{w_1\pi_1+w_2\pi_2}{w_{\text{com}}}$. Note that the term $\bar{w}q$ in equation \eqref{eq:theorem1} corresponds to the weighted sum of residual particles $\pi_{\text{res}}$. Setting $p_2$ to $\frac{\pi_{\text{com}}}{w_{\text{com}}}$ yields
\begin{equation} \label{eq:app_kl_bound_step1}
    D_{\text{KL}}(w_1\pi_1+w_2\pi_2+\pi_{\text{res}}, \pi_{\text{com}}+\pi_{\text{res}}) \le w_{\text{com}} D_{\text{KL}}\left(\frac{w_1\pi_1+w_2\pi_2}{w_{\text{com}}}, \frac{\pi_{\text{com}}}{w_{\text{com}}}\right).
\end{equation}
Next, by substituting $w$ with $\frac{w_1}{w_{\text{com}}}$, $q_1$ with $\pi_1$, $q_2$ with $\pi_2$, and $p$ with $\frac{\pi_{\text{com}}}{w_{\text{com}}}$ into equation \eqref{eq:theorem2a} of Lemma~\ref{thm:lemma_2}, we obtain
\begin{equation} \label{eq:app_kl_bound_step2}
    D_{\text{KL}}\left(\frac{w_1\pi_1+w_2\pi_2}{w_{\text{com}}}, \frac{\pi_{\text{com}}}{w_{\text{com}}}\right) \le \frac{w_1}{w_{\text{com}}} D_{\text{KL}}\left(\pi_1, \frac{\pi_{\text{com}}}{w_{\text{com}}}\right) + \frac{w_2}{w_{\text{com}}} D_{\text{KL}}\left(\pi_2, \frac{\pi_{\text{com}}}{w_{\text{com}}}\right).
\end{equation}
Combining equations \eqref{eq:app_kl_bound_step1} and \eqref{eq:app_kl_bound_step2} derives the desired upper bound for the KL divergence, denoted as $\mathcal{B}(\mathbf{x}_1,\mathbf{x}_2)$.
\begin{equation}
\label{eq:app_kl_upper_bound}
\begin{aligned}
    &D_{\text{KL}}
    (w_1\pi_1+w_2\pi_2+\pi_{\text{res}}, \pi_{\text{com}}+\pi_{\text{res}})
    \\
    &\quad\le
    w_1 D_{\text{KL}} \left(\pi_1, \frac{\pi_{\text{com}}}{w_{\text{com}}}\right)+w_2 D_{\text{KL}}\left(\pi_2,\frac{\pi_{\text{com}}}{w_{\text{com}}}\right)
    \\
    &\quad:=\mathcal{B}(\mathbf{x}_1,\mathbf{x}_2).
\end{aligned}
\end{equation}
Substituting equation \eqref{eq:gaussian_kl} from Lemma~\ref{thm:lemma_3} into $\mathcal{B}(\mathbf{x}_1,\mathbf{x}_2)$ results in
\begin{equation}
\label{eq:app_kl_bound_expanded}
\begin{aligned}
    2\mathcal{B}(\mathbf{x}_1,\mathbf{x}_2)
    &=w_1\bigg[\mathrm{tr}\bigg(\mathbf{\Sigma}_{\text{com}}^{-1}\bigg[\mathbf{\Sigma}_1 -\mathbf{\Sigma}_{\text{com}} +(\boldsymbol{\mu}_1-\boldsymbol{\mu}_{\text{com}}) (\boldsymbol{\mu}_1-\boldsymbol{\mu}_{\text{com}})^T\bigg]\bigg)
    \\
    &\qquad\qquad+\log\left(\frac{\det(\mathbf{\Sigma}_{\text{com}})}{\det(\mathbf{\Sigma}_1)}\right)\bigg]
    \\
    &\quad+w_2\bigg[\mathrm{tr}\bigg(\mathbf{\Sigma}_{\text{com}}^{-1}\bigg[\mathbf{\Sigma}_2 - \mathbf{\Sigma}_{\text{com}}+(\boldsymbol{\mu}_2-\boldsymbol{\mu}_{\text{com}})(\boldsymbol{\mu}_2-\boldsymbol{\mu}_{\text{com}})^T\bigg]\bigg)
    \\
    &\qquad\qquad+\log\left(\frac{\det(\mathbf{\Sigma}_{\text{com}})}{\det(\mathbf{\Sigma}_2)}\right)\bigg].
\end{aligned}
\end{equation}
Using the linearity of the trace operator, we can consolidate all trace-related terms from \eqref{eq:app_kl_bound_expanded} into the single expression, which is shown in Claim 1 below to be zero.

\begin{claim}
The following trace term is zero:
\begin{align*}
    \mathrm{tr}\Biggl(\mathbf{\Sigma}_{\text{com}}^{-1}\Biggl[ &w_1\left((\mathbf{\Sigma}_1-\mathbf{\Sigma}_{\text{com}}) + (\boldsymbol{\mu}_1-\boldsymbol{\mu}_{\text{com}})(\boldsymbol{\mu}_1-\boldsymbol{\mu}_{\text{com}})^T\right) \\
    & + w_2\left((\mathbf{\Sigma}_2-\mathbf{\Sigma}_{\text{com}}) + (\boldsymbol{\mu}_2-\boldsymbol{\mu}_{\text{com}})(\boldsymbol{\mu}_2-\boldsymbol{\mu}_{\text{com}})^T\right)\Biggr]\Biggr) = 0.
\end{align*}
\end{claim}

\begin{proof}

Let $\mathbf{A}$ denote the matrix term inside the trace term in Claim 1:
\begin{multline} \label{eq:app_matrix_A_def}
    \mathbf{A} = w_1\left[(\mathbf{\Sigma}_1-\mathbf{\Sigma}_{\text{com}}) + (\boldsymbol{\mu}_1-\boldsymbol{\mu}_{\text{com}})(\boldsymbol{\mu}_1-\boldsymbol{\mu}_{\text{com}})^T\right] \\
    + w_2\left[(\mathbf{\Sigma}_2-\mathbf{\Sigma}_{\text{com}})+ (\boldsymbol{\mu}_2-\boldsymbol{\mu}_{\text{com}})(\boldsymbol{\mu}_2-\boldsymbol{\mu}_{\text{com}})^T\right]
\end{multline}
We will show that $\mathbf{A} = \mathbf{0}$, which implies $\mathrm{tr}(\mathbf{\Sigma}_{\text{com}}^{-1}\mathbf{A}) = 0$.

First, we simplify the mean difference terms. Recall the definition of the combined weight and mean \eqref{eq:app_combined_cov}:
\begin{equation*}
    w_{\text{com}} = w_1 + w_2, \quad \boldsymbol{\mu}_{\text{com}} = \frac{w_1 \boldsymbol{\mu}_1 + w_2 \boldsymbol{\mu}_2}{w_{\text{com}}}
\end{equation*}
Then, we have:
\begin{align}
    \boldsymbol{\mu}_1 - \boldsymbol{\mu}_{\text{com}} &= \boldsymbol{\mu}_1 - \frac{w_1 \boldsymbol{\mu}_1 + w_2 \boldsymbol{\mu}_2}{w_{\text{com}}} \nonumber \\
    &= \frac{(w_1 + w_2)\boldsymbol{\mu}_1 - w_1 \boldsymbol{\mu}_1 - w_2 \boldsymbol{\mu}_2}{w_{\text{com}}} \nonumber \\
    &= \frac{w_2(\boldsymbol{\mu}_1 - \boldsymbol{\mu}_2)}{w_{\text{com}}}
\end{align}
\begin{align}
    \boldsymbol{\mu}_2 - \boldsymbol{\mu}_{\text{com}} &= \boldsymbol{\mu}_2 - \frac{w_1 \boldsymbol{\mu}_1 + w_2 \boldsymbol{\mu}_2}{w_{\text{com}}} \nonumber \\
    &= \frac{(w_1 + w_2)\boldsymbol{\mu}_2 - w_1 \boldsymbol{\mu}_1 - w_2 \boldsymbol{\mu}_2}{w_{\text{com}}} \nonumber \\
    &= \frac{-w_1(\boldsymbol{\mu}_1 - \boldsymbol{\mu}_2)}{w_{\text{com}}}
\end{align}
Let $\boldsymbol{\delta}_{12} = (\boldsymbol{\mu}_1 - \boldsymbol{\mu}_2)$. Substituting these back into the mean difference terms in equation~\eqref{eq:app_matrix_A_def}:
\begin{align}
    w_1 (\boldsymbol{\mu}_1 - \boldsymbol{\mu}_{\text{com}})(\boldsymbol{\mu}_1 - \boldsymbol{\mu}_{\text{com}})^T &= w_1 \left( \frac{w_2}{w_{\text{com}}} \boldsymbol{\delta}_{12} \right) \left( \frac{w_2}{w_{\text{com}}} \boldsymbol{\delta}_{12} \right)^T = \frac{w_1 w_2^2}{w_{\text{com}}^2} \boldsymbol{\delta}_{12}\boldsymbol{\delta}_{12}^T \\
    w_2 (\boldsymbol{\mu}_2 - \boldsymbol{\mu}_{\text{com}})(\boldsymbol{\mu}_2 - \boldsymbol{\mu}_{\text{com}})^T &= w_2 \left( \frac{-w_1}{w_{\text{com}}} \boldsymbol{\delta}_{12} \right) \left( \frac{-w_1}{w_{\text{com}}} \boldsymbol{\delta}_{12} \right)^T = \frac{w_2 w_1^2}{w_{\text{com}}^2} \boldsymbol{\delta}_{12}\boldsymbol{\delta}_{12}^T
\end{align}
Summing these two parts yields:
\begin{align}
    \left( \frac{w_1 w_2^2 + w_2 w_1^2}{w_{\text{com}}^2} \right) \boldsymbol{\delta}_{12}\boldsymbol{\delta}_{12}^T = \frac{w_1 w_2 (w_2 + w_1)}{w_{\text{com}}^2} \boldsymbol{\delta}_{12}\boldsymbol{\delta}_{12}^T = \frac{w_1 w_2}{w_{\text{com}}} \boldsymbol{\delta}_{12}\boldsymbol{\delta}_{12}^T
\end{align}
Now, we can rewrite the matrix $\mathbf{A}$:
\begin{align}
    \mathbf{A} &= (w_1 \mathbf{\Sigma}_1 + w_2 \mathbf{\Sigma}_2) - w_{\text{com}}\mathbf{\Sigma}_{\text{com}} + \frac{w_1 w_2}{w_{\text{com}}}\boldsymbol{\delta}_{12}\boldsymbol{\delta}_{12}^T
\end{align}
Finally, we utilize the definition of the combined covariance $\mathbf{\Sigma}_{\text{com}}$:
\begin{equation}
    \mathbf{\Sigma}_{\text{com}} = \frac{1}{w_{\text{com}}} \left( w_1 \mathbf{\Sigma}_1 + w_2 \mathbf{\Sigma}_2 + \frac{w_1 w_2}{w_{\text{com}}} \boldsymbol{\delta}_{12}\boldsymbol{\delta}_{12}^T \right)
\end{equation}
Multiplying both sides by $w_{\text{com}}$, we get:
\begin{equation} \label{eq:app_wcom_Sigmacom}
    w_{\text{com}}\mathbf{\Sigma}_{\text{com}} = w_1 \mathbf{\Sigma}_1 + w_2 \mathbf{\Sigma}_2 + \frac{w_1 w_2}{w_{\text{com}}} \boldsymbol{\delta}_{12}\boldsymbol{\delta}_{12}^T
\end{equation}
Substituting \eqref{eq:app_wcom_Sigmacom} into the expression for $\mathbf{A}$:
\begin{equation}
    \mathbf{A} = \left( w_1 \mathbf{\Sigma}_1 + w_2 \mathbf{\Sigma}_2 + \frac{w_1 w_2}{w_{\text{com}}} \boldsymbol{\delta}_{12}\boldsymbol{\delta}_{12}^T \right) - w_{\text{com}}\mathbf{\Sigma}_{\text{com}} = \mathbf{0}
\end{equation}
Since $\mathbf{A} = \mathbf{0}$, it follows immediately that $\mathrm{tr}(\mathbf{\Sigma}_{\text{com}}^{-1}\mathbf{A}) = 0$. 
\end{proof}

Therefore, $\mathcal{B}(\mathbf{x}_1,\mathbf{x}_2)$ simplifies to:
\begin{equation*}
    \mathcal{B}(\mathbf{x}_1,\mathbf{x}_2) = \frac{1}{2}\left[(w_1+w_2)\log \det(\mathbf{\Sigma}_{\text{com}}) - w_1\log \det(\mathbf{\Sigma}_1) - w_2\log \det(\mathbf{\Sigma}_2)\right].
\end{equation*}
This proves \eqref{eq:kl_bound_final}, and together with \eqref{eq:app_kl_upper_bound}, proves \cref{eq:kl_bound1}.










\end{document}